\documentclass[a4paper,final,english]{artikel} 

\title{Bernoulli crossed products without almost periodic weights}
\author{Peter Verraedt\thanks{Supported by a Ph.~D.~fellowship of the Research Foundation -- Flanders (FWO).}}
\date{}
\address{KU~Leuven, Department of Mathematics, Leuven (Belgium), peter.verraedt@wis.kuleuven.be}

\version{}

\usepackage{amsmath}
\usepackage{amssymb}
\usepackage{amsthm}
\usepackage{hyperref}
\usepackage{babel}
\usepackage{ifdraft}
\usepackage[capitalize]{cleveref} 
\usepackage[shortlabels]{enumitem}

\setlist{itemsep=1pt, topsep=5pt, partopsep=0pt}
\AtBeginDocument{%
\setlength{\abovedisplayskip}{6pt}%
\setlength{\belowdisplayskip}{6pt}%
}

\hypersetup{pdfnewwindow=true,colorlinks=true,linkcolor=black,citecolor=black,filecolor=black,urlcolor=black}

\newtheorem{thmstar}{Theorem}
\Crefname{thmstar}{Theorem}{Theorems}

\Crefname{corstar}{Corollary}{Corollaries}

\newtheorem{theorem}{Theorem}[section]
\Crefname{theorem}{Theorem}{Theorems}

\newtheorem{lemma}[theorem]{Lemma}
\Crefname{lemma}{Lemma}{Lemmas}

\Crefname{proposition}{Proposition}{Propositions}

\Crefname{corollary}{Corollary}{Corollaries}

\theoremstyle{definition}
\newtheorem{definition}[theorem]{Definition}
\Crefname{definition}{Definition}{Definitions}

\newtheorem{remark}[theorem]{Remark}
\Crefname{remark}{Remark}{Remarks}

\newcommand{\I}{\operatorname{I}}
\newcommand{\II}{\operatorname{I{\kern -0.2ex}I}}
\newcommand{\III}{\operatorname{I{\kern -0.2ex}I{\kern -0.2ex}I}}

\newcommand{\Inn}{\operatorname{Inn}}
\newcommand{\C}{\mathbb{C}}
\newcommand{\F}{\mathbb{F}}
\newcommand{\R}{\mathbb{R}}
\newcommand{\N}{\mathbb{N}}
\newcommand{\Rdual}{{\mathbb{R}^+_0}}
\newcommand{\Rhat}{\hat{\mathbb{R}}}

\newcommand{\actson}{\curvearrowright}

\newcommand{\Aut}{\operatorname{Aut}}
\newcommand{\Out}{\operatorname{Out}}
\newcommand{\lspan}{\operatorname{span}}
\newcommand{\T}{\mathbb{T}}
\newcommand{\Z}{\mathbb{Z}}
\newcommand{\cF}{\mathcal{F}}

\newcommand{\id}{\mathord{\operatorname{id}}}

\newcommand{\vphi}{\varphi}

\newcommand{\Tr}{\operatorname{Tr}}

\newcommand{\ovt}{\mathbin{\overline{\otimes}}}
\newcommand{\bigovt}{\mathbin{\overline\bigotimes}}

\newcommand{\cZ}{\mathcal{Z}}
\newcommand{\cK}{\mathcal{K}}

\newcommand{\ot}{\otimes}

\newcommand{\Ad}{\operatorname{Ad}}
\newcommand{\dpr}{^{\prime\prime}}

\newcommand{\Ntil}{\widetilde{N}}

\newcommand{\cN}{\mathcal{N}}

\newcommand{\cC}{\mathcal{C}}

\newcommand{\fK}{\mathfrak{K}}
\newcommand{\Ptil}{\widetilde{P}}

\newcommand{\inp}[2]{\langle#1, #2\rangle}
\newcommand{\fdot}{\mathop{\cdot}}

\newcommand{\defeq}{\mathrel{\mathop:}=}

\newcommand{\ap}{\text{ap}}

\renewcommand{\P}{P}
\renewcommand{\hat}{\widehat}
\renewcommand{\tilde}{\widetilde}

\begin{document}
\maketitle

\begin{abstract}
We prove a classification result for a large class of noncommutative Bernoulli crossed products $(P,\phi)^\Lambda \rtimes \Lambda$ without almost periodic states. Our results improve the classification results from \cite{vaes-verraedt;classification-type-III-bernoulli-crossed-products}, where only Bernoulli crossed products built with almost periodic states could be treated. 
We show that the family of factors $(P,\phi)^\Lambda \rtimes \Lambda$ with $P$ an amenable factor, $\phi$ a weakly mixing state (i.e.~a state for which the modular automorphism group is weakly mixing) and $\Lambda$ belonging to a large class of groups, is classified by the group $\Lambda$ and the action $\Lambda \curvearrowright (P, \phi)^\Lambda$, up to state preserving conjugation of the action.
\end{abstract}

\section{Introduction}
Distinguishing different von Neumann algebras is one of the main research goals in operator algebra theory. 
One of the main achievements in this direction was made by Popa, whose deformation/rigidity theory provided powerful classification results for type $\II_1$ factors such as group measure space constructions $L^\infty(X)\rtimes \Gamma$ of free ergodic probability measure preserving actions $\Gamma \curvearrowright (X,\mu)$. 
Popa's deformation/rigidity theory has later been combined with modular theory to study type $\III$ factors, such as group measure space constructions for nonsingular actions \cite{vaes-houdayer;type-III-unique-cartan}, Shlyakhtenko's free Araki-Woods factors \cite{houdayer;structural-results-free-araki-woods} and free quantum group factors \cite{isono;prime-factorization-free-quantum-group-factors}. In the same spirit, Stefaan Vaes and the author provided a classification result for noncommutative Bernoulli crossed products constructed with almost periodic states \cite{vaes-verraedt;classification-type-III-bernoulli-crossed-products}. 

However, obtaining a full classification for families of type $\III$ factors without using the existence of almost periodic states, is extremely difficult. Even the free Araki-Woods factors are only fully classified when the underlying one-parameter group is almost periodic \cite{shlyakhtenko;free-quasi-free-states;pacific}. Likewise, the classification of noncommutative Bernoulli crossed products in \cite{vaes-verraedt;classification-type-III-bernoulli-crossed-products} depends heavily on the fact that they have a discrete decomposition $N \rtimes \Gamma$, with $\Gamma$ a discrete group acting on a type $\II_\infty$ factor $N$. To go beyond almost periodic states, essentially new techniques are required. In the current paper, we provide a new argument showing classification results for Bernoulli crossed products with states that are not almost periodic. 

The noncommutative Bernoulli crossed products, introduced by Connes \cite{connes;almost-periodic-states}, are constructed as follows.
Let $(\P,\phi)$ be any von Neumann algebra equipped with a normal faithful state $\phi$, and let $\Lambda$ be a countably infinite group. Then take the infinite tensor product $\P^\Lambda = \bigovt_{g \in \Lambda} \P$ indexed by $\Lambda$, with respect to the state $\phi$. The group $\Lambda$ acts on $\P^\Lambda$ by shifting the tensor factors. This action is called a \emph{noncommutative Bernoulli action}, and the factor $\P^\Lambda \rtimes \Lambda$ is called a \emph{Bernoulli crossed product}. In \cite{vaes-verraedt;classification-type-III-bernoulli-crossed-products} it was shown that if $\phi$ is an almost periodic state, then the factors $(P,\phi)^{\F_n} \rtimes \F_n$ are completely classified up to isomorphism by $n$ and by the subgroup of $\Rdual$ generated by the point spectrum of the modular operator $\Delta_\phi$.

In this paper, we are able to improve this classification and to also include Bernoulli crossed products for which the state $\phi$ is not almost periodic. In this setting, the Bernoulli crossed product $P^\Lambda \rtimes \Lambda$ is always of type $\III_1$ (see \cref{lem.only-type-III1,lem.outer-core} below), and does not admit any almost periodic weight (see \cref{rem.no-almost-periodic}). Our classification thus yields many nonisomorphic factors of type $\III_1$ without almost periodic weights. The first result that we obtain, is a nonisomorphism result for Bernoulli crossed products $(P,\phi)^\Lambda \rtimes \Lambda$ where the state $\phi$ is weakly mixing, i.e.~the modular operator $\Delta_\phi$ of the state $\phi$ has no nontrivial finitely dimensional invariant subspaces.
We denote by $\cC$ the class of groups defined in \cite[Section 4]{vaes-verraedt;classification-type-III-bernoulli-crossed-products}. 
All groups in the class $\cC$ are nonamenable, and the class $\cC$ contains all weakly amenable groups $\Gamma$ with $\beta_1^{(2)}(\Gamma) > 0$ \cite{popa-vaes;unique-cartan-decomposition-factors-free}, all weakly amenable, nonamenable, bi-exact groups \cite{popa-vaes;unique-cartan-decomposition-factors-hyperbolic} and all free product groups $\Gamma = \Gamma_1 \star \Gamma_2$ with $|\Gamma_1| \geq 2$ and $|\Gamma_2| \geq 3$ \cite{ioana;cartan-subalgebras-amalgamated-free-product-factors}. 
In particular, it contains all nonelementary hyperbolic groups, such as the free groups $\F_n, n\geq2$.
Moreover, class $\cC$ is closed under extensions and commensurability. 
\begin{thmstar}\label{corstar.B}
The set of factors
\begin{align*}
\bigl\{(P,\phi)^\Lambda \rtimes \Lambda \bigm| &\;\;\text{$P$ a nontrivial amenable factor with a normal faithful weakly}\\
&\;\;\text{mixing state $\phi$, and $\Lambda$ an icc group in the class $\cC$}\;\bigr\}
\end{align*}
is exactly classified, up to isomorphism, by the group $\Lambda$ and the action $\Lambda \actson (P,\phi)^\Lambda$ up to a state preserving conjugacy of the action.
\end{thmstar}
In particular, \cref{corstar.B} implies that the state $\phi^\Lambda$ on $P^\Lambda$ can be retrieved from the factor $(P,\phi)^\Lambda \rtimes \Lambda$.

More generally, we also obtain the following optimal classification result for general states $\phi$ on the base algebra.
For every nontrivial factor $(P, \phi)$ equipped with a normal faithful state, we denote by $P_{\phi,\ap} \subset P$ the \emph{almost periodic part} of $P$, i.e.~$P_{\phi,\ap}$ is the subalgebra spanned by the eigenvectors of $\Delta_\phi$,
\begin{align*}
P_{\phi,\ap} = \big(\lspan \bigcup_{\mu \in \Rdual} \{ x \in P \mid \sigma_t^\phi(x) = \mu^{\mathbf{i}t} x \}\big)''. 
\end{align*}
We obtain the following main theorem, classifying all Bernoulli crossed products with amenable factors $(P,\phi)$ as base algebra, under the assumption that the almost periodic part of the base algebra, $P_{\phi,\ap}$, is a factor. This assumption is equivalent to the assumption that the centralizer $(P^I)_{\phi^I}$ of the infinite tensor product $P^I$ w.r.t.~$\phi^I$ is a factor, see \cref{lem.technical-condition} below. 

\begin{thmstar}\label{thmstar.A}
Let $(P_0,\phi_0)$ and $(P_1,\phi_1)$ be nontrivial amenable factors equipped with normal faithful states, such that $(P_0)_{\phi_0,\ap}$ and $(P_1)_{\phi_1,\ap}$ are factors. Let $\Lambda_0$ and $\Lambda_1$ be icc groups in the class $\cC$. 

The algebras $P_0^{\Lambda_0} \rtimes \Lambda_0$ and $P_1^{\Lambda_1} \rtimes \Lambda_1$ are isomorphic if and only if one of the following statements holds.
\begin{enumerate}[\upshape (a),topsep=0mm]
\item The states $\phi_0$ and $\phi_1$ are both tracial, and the actions $\Lambda_i \curvearrowright (P_i, \phi_i)^{\Lambda_i}$ are cocycle conjugate, modulo a group isomorphism $\Lambda_0 \cong \Lambda_1$.
\item The states $\phi_0$ and $\phi_1$ are both nontracial, and there exist projections $p_i \in (P_i^{\Lambda_i})_{\phi_i^{\Lambda_i}}$ such that the reduced cocycle actions $(\Lambda_i \curvearrowright (P_i, \phi_i)^{\Lambda_i})^{p_i}$ are cocycle conjugate through a state preserving isomorphism, modulo a group isomorphism $\Lambda_0 \cong \Lambda_1$.
\end{enumerate}
\end{thmstar}
Remark that the actions $\Lambda_i \curvearrowright P_i^{\Lambda_i}$ are outer, and under the conditions of the theorem, the centralizer of $P_i^{\Lambda_i}$ w.r.t.~$\phi_i^{\Lambda_i}$ is a factor. Hence, the reduced cocycle action $(\Lambda_i \curvearrowright (P_i, \phi_i)^{\Lambda_i})^{p_i}$ makes sense, and is well defined up to cocycle conjugation, see \cref{section.cocycle-actions}. 

Note that if the centralizers of $P_i^{\Lambda_i}$ w.r.t.~$\phi_i^{\Lambda_i}$ are trivial, then we automatically get conjugation of the two actions. 
In particular, \Cref{corstar.B} is now a direct consequence of \cref{thmstar.A}, since if $\phi$ is a weakly mixing state on $P$, then $P_{\phi,\ap} = \C$ is the trivial factor, and hence the centraliser of $P_i^{\Lambda_i}$ w.r.t.~$\phi_i^{\Lambda_i}$ is also trivial.
More generally, if the group $\Lambda_i$ is a direct product of two icc groups in the class $\cC$, we can apply Popa's cocycle superrigidity theorems and also get conjugation. We obtain the following result, analogous to Theorem B in \cite{vaes-verraedt;classification-type-III-bernoulli-crossed-products}.

\begin{thmstar}\label{corstar.C}
The set of factors
\begin{align*}
\bigl\{(P,\phi)^\Lambda \rtimes \Lambda \bigm| &\;\;\text{$P$ a nontrivial amenable factor with normal faithful state $\phi$}\\
&\;\;\text{such that $P_{\phi,\ap}$ is a factor, and $\Lambda$ a direct product of two icc}\\
&\;\;\text{groups in the class $\cC$}\;\bigr\}
\end{align*}
is exactly classified, up to isomorphism, by the group $\Lambda$ and the action $\Lambda \actson (P,\phi)^\Lambda$ up to a state preserving conjugacy of the action.
\end{thmstar}

For \emph{two-sided} Bernoulli actions $\Lambda \times \Lambda \curvearrowright P^\Lambda$, we get the following improvement of Theorem C in \cite{vaes-verraedt;classification-type-III-bernoulli-crossed-products}. Note that even in the almost periodic case, we obtain a strengthening, as we can remove the condition that $\Lambda$ has no nontrivial \emph{central sequences}. A central sequence in a group $\Lambda$ is a sequence $g_n \in \Lambda$ such that for all $h \in \Lambda$, $g_n h = h g_n$ eventually. The absence of such sequences ensures that the crossed product $M = P^\Lambda \rtimes (\Lambda \times \Lambda)$ is a full factor, i.e.~that $\Inn(M) \subset \Aut(M)$ is closed for the topology given by $\alpha_n \to \alpha$ if and only if $\forall \psi \in M_\star : \|\psi \circ \alpha_n -\psi \circ \alpha\| \to 0$ (see \cite{connes;almost-periodic-states}). 
Since we do not require that $\Lambda$ has no nontrivial central sequences, the following result possibly yields a class of nonisomorphic type $\III_1$ factors that are not full. However, there are probably no groups with central sequences in class $\cC$.
\begin{thmstar}\label{corstar.D}
The set of factors
\begin{align*}
\bigl\{(P,\phi)^\Lambda \rtimes (\Lambda \times \Lambda) \bigm| &\;\;\text{$P$ a nontrivial amenable factor with normal faithful state $\phi$}\\
&\;\;\text{such that $P_{\phi,\ap}$ is a factor, and $\Lambda$ an icc group in the class $\cC$}\;\bigr\}
\end{align*}
is exactly classified, up to isomorphism, by the group $\Lambda$ and the pair $(P,\phi)$ up to a state preserving isomorphism.
\end{thmstar}

The proofs of the above nonisomorphism results all make use of Tomita-Takesaki's modular theory and the continuous core of type $\III$ factors. We first use the main results of \cite{popa-vaes;unique-cartan-decomposition-factors-free,popa-vaes;unique-cartan-decomposition-factors-hyperbolic,ioana;cartan-subalgebras-amalgamated-free-product-factors} providing unique crossed product decomposition theorems for factors of the form $R \rtimes \Lambda$, where $\Lambda \actson R$ is an outer action on the hyperfinite $\II_1$ factor $R$ of a group $\Lambda$ in class $\cC$. We then obtain a cocycle conjugation isomorphism $\psi : P_0^\Lambda \rtimes \R \to P_1^\Lambda \rtimes \R$ between the induced actions $\Lambda \curvearrowright P_i^\Lambda \rtimes \R$ on the continuous cores. Using the spectral gap methods of \cite{popa;superrigidity-malleable-actions-spectral-gap}, we obtain mutual intertwining $\psi(L\R) \prec L\R$, $L\R \prec \psi(L\R)$, see \cref{lem.pointwise-fixed-subalgebra-intertwines-in-R} below. Under the assumption that the almost periodic parts $(P_i)_{\phi_i,\ap}$ are factors, we are able to deduce that the actions $\Lambda \curvearrowright P_i^\Lambda$ must then be cocycle conjugate, up to reductions, see \cref{lem.pointwise-fixed-subalgebra-intertwines-in-R-conclusion,lem.state-preserving-actions}.

Recently, Houdayer and Isono \cite{houdayer-isono;unique-prime-factorization} introduced intertwining techniques in the general type $\III$ setting. While it is tempting to use these interesting techniques to prove the nonisomorphism results in the current article, there are two obstructions to do so. Firstly, the main intertwining theorems for group actions on type $\II_1$ factors \cite{popa-vaes;unique-cartan-decomposition-factors-free,popa-vaes;unique-cartan-decomposition-factors-hyperbolic,ioana;cartan-subalgebras-amalgamated-free-product-factors} do currently not have general analogues in the type $\III$ setting. Secondly, it is unclear whether a direct type $\III$ approach would allow to recover a state preserving conjugacy, as in the above theorems. After all, in our situation, we do have `favorite' states, in contrast to the unique prime factorization problem studied in \cite{houdayer-isono;unique-prime-factorization}.

\section{Preliminaries}

All mentioned von Neumann algebras are assumed to have separable predual.

\subsection{Cocycle actions}\label{section.cocycle-actions}
Let $M$ be a von Neumann algebra equipped with a normal semifinite faithful (n.s.f.) weight $\varphi$. 
We denote by $\Aut(M)$ the group of automorphisms of $M$, and by $\Aut(M, \vphi)$ the subgroup of weight scaling automorphisms of $M$, i.e.~automorphisms $\alpha$ for which there exists a $\lambda \in \R^+_0$ such that $\vphi \circ \alpha = \lambda \vphi$. Equipped with the topology where a net $\alpha_n \in \Aut(M)$ converges to $\alpha$ if and only if for every $\psi \in M_\star$, $\| \psi \circ \alpha_n - \psi \circ \alpha\|$ converges to zero, $\Aut(M)$ and $\Aut(M,\vphi)$ become Polish groups. 

Let $G$ now be a locally compact group. A \emph{cocycle action} of $G$ on $(M,\varphi)$ is a continuous map
$ \alpha : G \to \Aut(M,\varphi) : g \mapsto \alpha_g$ and a continuous map $v : G \times G \to \mathcal{U}(M_\varphi)$ such that:
\begin{align*}
\alpha_e = \text{id}, \qquad \alpha_{g} \circ \alpha_{h} &= \Ad v_{g,h} \circ \alpha_{gh}, \qquad\forall g,h,k\in G, \\
v_{g,h}v_{gh,k} &= \alpha_g (v_{h,k}) v_{g,hk}.
\end{align*}
Here, $M_\vphi$ denotes the \emph{centraliser} of $\varphi$.
Such an action is denoted by $G \curvearrowright^{\alpha,v} (M,\varphi)$. A strongly continuous map $v$ satisfying the above relations is called a 2-\emph{cocycle} for $\alpha$. If $v=1$, $\alpha$ is called an \emph{action}.

Two cocycle actions $G \curvearrowright^{\alpha,v} (M_1,\vphi_1)$ and $G \curvearrowright^{\beta,w} (M_2,\vphi_2)$ are \emph{cocycle conjugate through a weight preserving isomorphism} if there exists a strongly continuous map $u : G \to \mathcal{U}( (M_2)_{\vphi_2})$ and an isomorphism $\psi : (M_1,\vphi_1) \to (M_2,\vphi_2)$ satisfying $\vphi_2 \circ \psi = \vphi_1$ and 
\begin{align*}
\psi \circ \alpha_g &= \Ad u_g \circ \beta_g \circ \psi,\qquad\forall g,h \in G, \\
\psi(v_{g,h})&=u_g \beta_g(u_h) w_{g,h} u^\star_{gh}.
\end{align*}
If $\delta : G_1 \to G_2$ is a continuous isomorphism, we say that the cocycle actions $G_1 \curvearrowright^{\alpha,v} (M_1,\varphi_1)$ and $G_2 \curvearrowright^{\beta,w} (M_2,\varphi_2)$ are cocycle conjugate \emph{modulo $\delta$}, if the actions $\alpha$ and $\beta \circ \delta$ of $G_1$ are cocycle conjugate.

An automorphism $\alpha$ of $M$ is called \emph{properly outer} if there exists no nonzero element $y \in M$ such that $y \alpha(x)=xy$ for all $x \in M$. A cocycle action $\Gamma \curvearrowright^{\alpha,v} (M,\varphi)$ of a discrete group is called \emph{properly outer} if $\alpha_g$ is properly outer for all $g \in \Gamma, g\not=e$.

Assume now that $(M,\varphi)$ is a von Neumann algebra with an n.s.f.~weight for which $M_\varphi$ is a factor. Consider a properly outer cocycle action $\Gamma \curvearrowright^{\alpha,v}(M,\varphi)$ of a discrete group $\Gamma$, that preserves the weight $\varphi$.
Suppose that $p \in M_\varphi$ is a nonzero projection with $\varphi(p) < \infty$, and choose partial isometries $w_g \in M_\varphi$ such that $p = w_gw_g^\star$, $\alpha_g(p) = w_g^\star w_g$, $w_e=p$.
Let $\alpha^p : G \to \Aut(pM p, \varphi_p)$ be defined by
$ \alpha_g^p(pxp) = w_g\alpha_g(pxp)w_g^\star$, for $x \in M$, where $\varphi_p(pxp) = \varphi(pxp)/\varphi(p)$. Denote $v_{g,h}^p = w_g \alpha_g(w_h) v_{g,h} w_{gh}^\star \in pM_\varphi p$, for $g,h\in G$. Then $G \curvearrowright^{\alpha^p,v^p}(pM p,\varphi_p)$ is a properly outer cocycle action, which does not depend on the choice of the partial isometries $w_g \in M_\varphi$, up to state preserving cocycle conjugacy. We call this action the \emph{reduced cocycle action of $\alpha$ by $p$}.

We say that two properly outer cocycle actions $\Gamma \curvearrowright^{\alpha,v} (M_1,\vphi_1)$ and $\Gamma \curvearrowright^{\beta,w} (M_2,\vphi_2)$ of a discrete group $\Gamma$ are, \emph{up to reductions}, cocycle conjugate through a state preserving isomorphism, if there exists projections $p_i \in (M_i)_{\vphi_i}$ such that the reduced cocycle actions $\alpha^{p_0}$ and $\beta^{p_1}$ are cocycle conjugate through a state preserving isomorphism.

\subsection{Popa's theory of intertwining-by-bimodules}

Let $(M,\tau)$ be any tracial von Neumann algebra, and let $P \subset 1_P M 1_P$ and $Q \subset 1_Q M 1_Q$ be von Neumann subalgebras. Following \cite{popa;strong-rigidity-malleable-actions-I}, we write $P \prec_M Q$ if there exist a projection $p \in M_n(\C) \ot Q$, a normal unital $\star$-homomorphism $\theta : P \to p(M_n(\C) \ot Q)p$ and a non-zero partial isometry $v \in M_{1,n}(\C) \ot 1_P M 1_Q$ satisfying $a v = v \theta(a)$ for all $a \in P$.

We recall from \cite{vaes-verraedt;classification-type-III-bernoulli-crossed-products} the definition of the class of groups $\cC$.

Let $(M,\tau)$ be a tracial von Neumann algebra. 
For every von Neumann subalgebra $P\subset M$, we consider the \emph{normaliser} $\mathcal{N}_M(P) = \{u \in \mathcal{U}(M) \mid uPu^\star = P\}$. 
The \emph{Jones index} of a von Neumann subalgebra $P \subset M$ is defined as the $P$-dimension of $L^2(M)$ as a $P$-module, computed using the given trace $\tau$.
An inclusion $P \subset M$ is said to be \emph{essentially of finite index} if there exist projections $p \in P' \cap M$ that lie arbitrarily close to $1$ such that $P p \subset pMp$ has finite Jones index.
We call $A \subset M$ a \emph{virtual core subalgebra} if $A' \cap M  = \cZ(A)$ and if the inclusion $\cN_M(A)\dpr \subset M$ is essentially of finite index.

\begin{definition}[\cite{vaes-verraedt;classification-type-III-bernoulli-crossed-products}]\label{definition:class-C}
We say that a countably infinite group $\Gamma$ belongs to class $\cC$ if for every trace preserving cocycle action $\Gamma \actson (B,\tau)$ and every amenable, virtual core subalgebra $A \subset p (B \rtimes \Gamma)p$, we have that $A \prec B$.
\end{definition}

Recall that all groups in the class $\cC$ are nonamenable, and that class $\cC$ contains all weakly amenable groups $\Gamma$ with $\beta_1^{(2)}(\Gamma) > 0$ \cite{popa-vaes;unique-cartan-decomposition-factors-free}, all weakly amenable, nonamenable, bi-exact groups \cite{popa-vaes;unique-cartan-decomposition-factors-hyperbolic} and all free product groups $\Gamma = \Gamma_1 \star \Gamma_2$ with $|\Gamma_1| \geq 2$ and $|\Gamma_2| \geq 3$ \cite{ioana;cartan-subalgebras-amalgamated-free-product-factors}. Moreover, class $\cC$ is closed under extensions and commensurability. 
We refer to section 4 of \cite{vaes-verraedt;classification-type-III-bernoulli-crossed-products} for more details.

We will also need the following elementary lemma.

\begin{lemma}
Let $N$ be a finite von Neumann algebra. If $A,B \subset N$ are abelian subalgebras with
$A = (A' \cap N)' \cap N$, $B = (B' \cap N)' \cap N$ and $A \prec_N B$, then there exists a nonzero partial isometry $w \in N$ with $ww^\star \in A' \cap N$, $w^\star w \in B' \cap N$ and $A w \subset w B$. 
\label{lemma:intertwining-abelian-subalg}
\end{lemma}
\begin{proof}
Let $A,B \subset N$ as in the statement of the lemma. Denote by $P$ and $Q$ the relative commutants of $A$ and $B$, i.e.~$P = A' \cap N$ and $Q = B' \cap N$.
Since $A \prec_N B$, we also have that $Q \prec_N P$, see \cite[Lemma 3.5]{vaes;explicit-computations-finite-index-bimodules}.  
Hence, we find projections $p \in P, q \in Q$, a partial isometry $v \in N$ with $vv^\star \leq q$, $v^\star v \leq p$ and a normal unital $\star$-homomorphism $\theta : qQq \to pPp$ satisfying 
\begin{align*}
x v = v \theta(x) \quad\text{ for all}\quad x \in qQq. 
\end{align*}
Note that $vv^\star \in (qQq)' \cap qNq = qB$. 
Write now $D = \theta(qQq)' \cap pNp$ and $f = v^\star v \in D$, then by spatiality, we have
\begin{align*}
f D f = (\theta(qQq)f)' \cap fNf = (v^\star Q v)' \cap fNf = v^\star (qB) v. 
\end{align*}
Hence $f$ is an abelian projection of $D$. Note that $pA \subset D$ is an abelian subalgebra.
Take now $C \subset D$ a maximal abelian subalgebra satisfying $pA \subset C$, and observe that necessarily $C \subset pPp$. Since $D$ is a finite von Neumann algebra, we find a partial isometry $v_1 \in D$ such that $v_1v_1^\star = f$ and $v_1^\star D v_1 \subset C$, see e.g.~\cite[Lemma C.2]{vaes;rigidity-results-bernoulli}. 
Put $w = vv_1$, then we still have
\begin{align*}
 xw = w \theta(x) \quad\text{ for all}\quad x \in qQq,
\end{align*}
since $v_1$ commutes with $\theta(qQq)$, and $ww^\star = vfv^\star = vv^\star \in qB \subset qQq$. Moreover, now we obtain $w^\star w = v_1^\star f v_1 \in C \subset pPp$. We get that $w^\star (qQq) w \subset pPp$, hence $ww^\star Q ww^\star \subset wPw^\star$. Taking the relative commutants in $ww^\star N ww^\star$, we obtain that $w A w^\star \subset ww^\star B$, hence $w^\star$ is the desired partial isometry.
\end{proof}

\subsection{Takesaki's modular theory}
Let $M$ be a von Neumann algebra equipped with a normal semifinite faithful (n.s.f.) weight $\varphi$. Denote by $\sigma^\vphi : \R \curvearrowright M$ the modular automorphism group associated to $\vphi$. The associated crossed product $M \rtimes \R$ of $M$ with the action $\sigma^\vphi$ is called the \emph{continuous core}. We denote by $\pi_\sigma : M \to M\rtimes \R$ the canonical embedding, and by $(\lambda(t))_{t \in \R}$ the canonical group of unitaries such that $\pi_\sigma(\sigma^\vphi_t(x)) = \lambda(t) \pi_\sigma(x)\lambda(t)^\star$. 

Consider $\mathbb{R}^+_0$ to be the dual of $\mathbb{R}$ under the pairing $\inp{t}{\mu}=\mu^{\mathbf{i}t}$ for $t \in \mathbb{R}$, $\mu \in \mathbb{R}^+_0$, and let $\hat{\sigma}^\vphi : \Rdual \curvearrowright M \rtimes \R$ be the dual action to $\sigma^\vphi$.
We have by Takesaki duality \cite[Theorem X.2.3]{takesaki;theory-operator-algebras-II} that
\begin{align*}
(M \rtimes \R) \rtimes \Rdual \cong M \ovt B(L^2(\R))
\end{align*}
by the isomorphism $\Phi : (M \rtimes \R) \rtimes \Rdual \to M \ovt B(L^2(\R))$ given by
\begin{align}
\{\Phi\big(\pi_{\hat\sigma} \circ \pi_\sigma(x)\big)\xi\}(t) &= (\sigma_t^\vphi)^{-1}(x)\xi(t),&&x \in M, \quad \xi \in L^2(\R, L^2(M)), \quad t \in \R,\nonumber\\
\{\Phi\big(\pi_{\hat\sigma} \circ \lambda(s)\big)\xi\}(t) &= \xi(t-s), &&s\in \R,\label{eq:takesakiduality}\\
\{\Phi\big(\lambda(\mu)\big)\xi\}(t) &=\overline{\inp{t}{\mu}}\xi(t), &&\mu \in \Rdual. \nonumber
\end{align}
Here $(\lambda(\mu))_{\mu \in \Rdual}$ is the canonical group of unitaries in the second crossed product.

There exists an n.s.f.~trace $\Tr_\varphi$ on $M\rtimes \R$ such that
\begin{align*}
\Tr_\varphi \circ \hat\sigma^\vphi_{\mu} = \mu^{-1} \Tr_\varphi  \text{ for all }\mu \in \Rdual,
\end{align*}
and such that the dual weight $\tilde\Tr_\varphi$ (see \cite[Definition X.1.16]{takesaki;theory-operator-algebras-II}) on $(M\rtimes \R)\rtimes \Rdual$ corresponds under the duality $\Phi$ to the weight $\varphi \otimes \Tr(h \fdot)$ on $M \ovt B(L^2(\R))$. Here $h$ is the hermitian operator given by $h^{\mathbf{i}t} = \lambda(t) \in B(L^2(\R))$, with $\lambda(t)$ the left regular representation defined by $\big(\lambda(t)f\big)(s) = f(s-t)$ for $f \in L^2(\R)$.

The following lemma is well known. For a proof, see e.g.~\cite[Proposition 2.4]{houdayer-ricard;free-araki-woods-factors}.
\begin{lemma}\label{lem.commutant}
Let $(M,\vphi)$ be a von Neumann algebra equipped with a n.s.f.~weight. Then $(L\R)' \cap M \rtimes_{\sigma^\vphi} \R = M_\vphi \ovt L\R$.
\end{lemma}

\subsection{Structural properties of infinite tensor products}
Let $(P,\phi)$ be a von Neumann algebra with a normal faithful state $\phi$. Whenever $I$ is a countable set, we write $P^I$ for the tensor product of $P$ indexed by $I$ with respect to $\phi$. The canonical product state on $P^I$ will be denoted by $\phi^I$. In this section, we study the structure of the infinite tensor product $(P^I, \phi^I)$. The first result we get, is a type classification for these infinite tensor products.
We show in particular that such tensor products never yield a factor of type $\III_0$. This result is probably well known, but we could not find a reference in the literature.
\begin{lemma}\label{lem.only-type-III1}
Let $(P,\phi)$ be a nontrivial factor equipped with a normal faithful state, and let $I$ be a countable infinite set. The factor $(P,\phi)^I$ is of type
\begin{itemize}[align=left,leftmargin=1.5em,topsep=0mm]\parskip2pt
\item[$\II_1$] if $\phi$ is tracial,
\item[$\III_{\lambda}, \lambda \in (0,1)$] if $\phi$ is nontracial and periodic with period $\frac{2\pi}{|\log \lambda|}$, and
\item[$\III_1$] if $\phi$ is not periodic.
\end{itemize}
Moreover, denoting by $P^I \rtimes \R$ the crossed product with the modular action of $\phi^I$, we have that $\mathcal{Z}(P^I \rtimes \R) = L(G)$, where $G < \R$ is the subgroup given by $G = \{t \in \R \mid \sigma^\phi_t = \id\}$.
\end{lemma}
\begin{proof}
Put $\vphi = \phi^I$. 
If $\phi$ is tracial, then clearly $\vphi$ is a trace on $(P,\phi)^I$, hence $(P,\phi)^I$ is a type $\II_1$ factor and clearly $\mathcal{Z}(P^I \rtimes \R) = L(\R)$. Assume now that $\phi$ is not tracial. Then $(P,\phi)^I$ cannot be semifinite and thus is a type $\III$ factor, see e.g.~\cite[Theorem XIV.1.4]{takesaki;theory-operator-algebras-III}. 

Identify $I = \N$ and let $N = P^\N \rtimes \R$ denote the crossed product w.r.t.~the modular group $\sigma^{\phi^I}$. For every $n \in \N$, let $\alpha_n$ denote the $\star$-isomorphism $N \to (P \ovt P^\N) \rtimes \R$ defined by
\begin{align*}
\alpha_n\big( \pi_\sigma(\ot_k x_k) \big) &= \pi_\sigma\big( x_n \ot \big(x_0 \ot x_1 \ot \cdots \ot x_{n-1} \ot x_{n+1} \ot \cdots \big)\big), &&\text{ for }x_k \in P, \\ 
\alpha_n( \lambda(t)) &= \lambda(t), &&\text{ for }t \in \R.
\end{align*}
Let $\iota : N \to (1 \ot P^\N) \rtimes \R \subset (P \ovt P^\N)\rtimes \R$ be the canonical embedding, and remark that for all $x \in N$, $\alpha_n(x) \to \iota(x)$ $\star$-strongly.
Then for every element $x \in \mathcal{Z}(N)$, we have that $\iota(x) \in \mathcal{Z}((P \ovt P^\N)\rtimes\R)$, and in particular $[\iota(x), \pi_\sigma(a \ot 1)] = 0$ for all $a \in P$. 

For any $\omega \in P_\star$, denote by $\text{ev}_\omega: L^2(\R, P \ovt P^\N) \to L^2(\R, P^\N)$ the linear map given by applying $\omega \ot \id$ to $P \ovt P^\N$ diagonally. Also denote by $\hat\iota : L^2(\R, P^\N) \to L^2(\R,P \ovt P^\N)$ the map induced by $\iota$. The above now means that for $a \in P$, $\text{ev}_\omega \circ \pi_\sigma(a \ot 1) \circ \hat\iota$ commutes with $\mathcal{Z}(N)$, when acting on $L^2(N)$. But note that 
\begin{align*}
\{\big(\text{ev}_\omega \circ \pi_\sigma(a \ot 1) \circ \hat\iota\big) \xi\}(s) = \omega(\sigma^\phi_{-s}(a))\xi(s)
\quad\text{for }\xi \in L^2(\R, L^2(P^\N)).
\end{align*}
This means that $\mathcal{Z}(N) \subset (1 \ot D)' \cap B(L^2(N))$, where $D \subset L^\infty(\R)$ is the von Neumann subalgebra generated by the functions $t \mapsto \omega(\sigma^\phi_{-t}(a))$ for $\omega \in P_\star, a \in P$. 
On the other hand, we have by \cref{lem.commutant} that $\mathcal{Z}(N) \subset (L\R)' \cap N = (P^\N)_\vphi \ovt L\R$. Combining both facts, we obtain that $\mathcal{Z}(N) \subset (P^\N)_\vphi \ovt L\R \cap (1 \ot D)' = (P^\N)_\vphi \ovt (L\R \cap D')$.

We now make the following case distinction. \textbf{Case 1:} The state $\phi$ is not periodic. In this case, the functions $t \mapsto \omega(\sigma^\phi_{-t}(a))$ for $\omega \in P_\star,a \in P$ separate points of $\R$, and hence we have that $D = L^\infty(\R)$. In particular, $L\R \cap D' = \C$, meaning that $\mathcal{Z}(N) \subset (P^\N)_\vphi$. As $P^\N$ is a factor, we conclude that $N$ is also a factor, and hence $P^\N$ is of type $\III_1$. 

\textbf{Case 2:} The state $\phi$ is periodic with period $T > 0$. Then $D$ consists of all bounded functions $f \in L^\infty(\R)$ that are $T$-periodic, i.e.~$f(s)=f(s+T)$ for all $s \in \R$. 
In particular, $D$ contains the function $f : t \mapsto \exp( \frac{2\pi \mathbf{i}t}{T})$. Using the Fourier transform $L\R \cong L^\infty(\Rhat)$, it is now easy to see that $L\R \cap D' \subset L\R \cap\{M_f\}' = L(T\Z)$, and hence $\mathcal{Z}(N) \subset (P^\N)_\vphi \ovt L(T\Z)$. Using the Fourier decomposition of elements in $P^\N \rtimes T\Z$, we now can conclude that $\mathcal{Z}(N) = L(T\Z)$, and hence $P^\N$ is a factor of type $\III_\lambda$ with $\lambda = e^{- \frac{2\pi}{T}}$. 
\end{proof}

For $\mu \in \Rdual$, we denote by $P_{\phi,\mu}$ the eigenvectors of $\Delta_\phi$ for $\mu$, i.e.~
\begin{align*}
P_{\phi,\mu} = \{ x \in P \mid \Delta_\phi \hat{x} = \mu \hat{x} \} = \{ x \in P \mid \sigma_\phi^t(x) = \mu^{\mathbf{i}t} x \}.
\end{align*}
Note that $\overline{P_{\phi,\mu}}^{\|\fdot\|_\phi} = \{ \xi \in L^2(P,\phi) \mid \Delta_\phi \xi = \mu \xi\}$. 
Recall that a state $\phi$ is called \emph{almost periodic} if $\Delta_\phi$ is diagonalizable. In general, we define $P_{\phi,\ap} \subset P$ as the subalgebra spanned by the eigenvectors of $\Delta_\phi$, i.e.~
\begin{align*}
P_{\phi,\ap} = \big(\lspan \bigcup_{\mu \in \Rdual} P_{\phi,\mu} \big)''. 
\end{align*}
The notation $P_{\phi,\ap}$ stands for the \emph{almost periodic part} of $P$, and this name is justified since $P_{\phi,\ap}$ is the maximal subalgebra $Q \subset P$ with \emph{$\phi$-preserving conditional expectation}, 
such that the restriction of $\phi$ to $Q$ is almost periodic. Here, a subalgebra $Q \subset P$  is said to be with $\phi$-preserving conditional expectation, if there exists a conditional expectation $E : P \to Q$ satisfying $\phi \circ E = \phi$.
In \cref{lem.technical-condition} below, we will show that if $(P,\phi)$ is a nontrivial factor equipped with a normal faithful state and $I$ is an infinite set, then the centralizer $(P^I)_{\phi^I}$ of the infinite tensor product is a factor if and only if $P_{\phi,\ap}$ is a factor. To show this, we need the following elementary lemma. 

\begin{lemma}
Let $(M,\psi)$, $(N,\vphi)$ be factors equipped with normal faithful states. For every $\mu \in \R^+_0$, it holds that
\begin{align}\label{eq:almost-periodic-part-tensor-product}
(M \ovt N)_{\psi \ot \vphi,\mu} \subset \bigoplus_{t \in \Rdual} \overline{M_{\psi,\mu t^{-1}} \ot N_{\vphi,t}}^{\|\fdot\|_{\psi \ot \vphi}},
\end{align}
as subsets of $L^2(M \ovt N, \psi\ot \vphi)$. In particular, 
$(M \ovt N)_{\psi \ot \vphi, \ap} =  M_{\psi,\ap} \ovt N_{\vphi,\ap}$.
\label{lem.almost-periodic-part-tensor-product}
\end{lemma}
\begin{proof}
Let $(M,\psi)$ and $(N,\vphi)$ be factors with normal faithful states, and fix $\mu \in \Rdual$. In this proof, we will write $L^2(M)$ and $L^2(N)$ for $L^2(M,\psi)$ and $L^2(N,\vphi)$ respectively. 
Assume that $x \in (M\ovt N)_{\psi \ovt \vphi, \mu}$, and consider $\hat{x} \in L^2(M)\ot L^2(N)$. Note that $(\Delta_\psi \ot \Delta_\vphi)\hat{x} = \mu \hat{x}$, hence $\hat{x} = \sum_{k \in \N} \xi_k \ot \eta_k$, where $\xi_k \in L^2(M)$ and $\eta_k \in L^2(N)$ are eigenvectors for $\Delta_\psi$ and $\Delta_\vphi$ respectively, such that the product of the eigenvalues of $\xi_k$ and $\eta_k$ equals $\mu$. 
Now \eqref{eq:almost-periodic-part-tensor-product} follows from the observation that for every von Neumann algebra $(P,\phi)$ equipped with a normal faithful state, and for every $\mu \in \R^+_0$, $\overline{P_{\phi,\mu}}^{\|\fdot\|_\phi} = \{\xi \in L^2(P) \mid \Delta_\phi \xi = \mu \xi\}$.

It follows immediately from \eqref{eq:almost-periodic-part-tensor-product} that $(M \ovt N)_{\psi \ot \vphi, \mu} \subset  L^2(M_{\psi,\ap} \ovt N_{\vphi,\ap}, \psi \ot \vphi)$, and hence for every $x \in (M\ovt N)_{\psi \ot \vphi,\mu}$ we get $E_{M_{\psi,\ap} \ovt N_{\vphi,\ap}}(x) = x$. This demonstrates that indeed $(M \ovt N)_{\psi \ot \vphi, \ap} =  M_{\psi,\ap} \ovt N_{\vphi,\ap}$.
\end{proof}

\begin{lemma}
Let $(P,\phi)$ be a nontrivial factor equipped with a normal faithful state, and let $I$ be a countable infinite set.
Then it holds that $(P^I)_{\phi^I}  = (P_{\phi,\ap}^I)_{\phi^I}$.
In particular, the following two statements are equivalent:
\begin{enumerate}[(i)]
\item $(P^I)_{\phi^I}$ is a factor.
\item $P_{\phi,\ap}$ is a factor.
\end{enumerate}
\label{lem.technical-condition}
\end{lemma}
\begin{proof}
Let $P$, $I$ be as in the statement, and put $Q = P_{\phi,\ap}$ and $\vphi = \phi^I$. The inclusion $(Q^I)_{\vphi} \subset (P^I)_{\vphi}$ being obvious, take $x \in (P^I)_{\vphi}$; we will show that $x \in Q^I$. Note that for all finite subsets $\cF \subset I$, $E_{P^\cF}(x) \in (P^\cF)_{\vphi}$, and by \cref{lem.almost-periodic-part-tensor-product}, we have that
\begin{align*}
E_{P^\cF}(x) \in (P^\cF)_{\vphi} \subset (P^\cF)_{\vphi,\ap} = Q^\cF \subset Q^I. 
\end{align*}
As $x$ is a limit point of $\{E_{P^\cF}(x) \mid \cF \subset I \text{ finite}\}$ in the strong topology, it follows that also $x \in Q^I$. 

The implication (ii) $\Rightarrow$ (i) follows now directly from \cite[Lemma 2.4]{vaes-verraedt;classification-type-III-bernoulli-crossed-products}. For the reverse implication, 
note that $\mathcal{Z}(P_{\phi,\ap}) \subset P_\phi$. Thus, if 
$x \in \mathcal{Z}(P_{\phi,\ap})$ is a nontrivial element, then $x \otimes 1 \otimes 1 \otimes \cdots$ belongs to $(P^I)_\vphi$, and commutes with all elements of $(P_{\phi,\ap})^I$. 
In particular, $(P^I)_\vphi = (P_{\phi,\ap}^I)_\vphi$ is not a factor.
\end{proof}

We recall the definition of the (generalized) Bernoulli action.
Let $(P,\phi)$ be a von Neumann algebra with a normal faithful state $\phi$, and let $I$ be a countable set.
Consider a countable group $\Lambda$ that acts on $I$, and let $\Lambda$ act on $\P^I$ by the Bernoulli action
\begin{align*}
\rho(s)\big(\otimes_{k \in I} a_k\big) &= \otimes_{k \in I} a_{s^{-1} \cdot k},&\text{for }s \in \Lambda, a_h \in \P.
\end{align*}
The von Neumann algebra $(P,\phi)$ is called the \emph{base algebra} for the Bernoulli action, and the crossed product $P^I \rtimes \Lambda$ is called the \emph{Bernoulli crossed product}.

We now study when a Bernoulli action $\Lambda \curvearrowright P^I$ and its associated action on the continuous core $\Lambda \curvearrowright P^I \rtimes \R$ are properly outer. 

\begin{remark}\label{remark.shift-outer}
Let $(P,\phi)$ be a nontrivial factor equipped with a normal faithful state, and let $I$ be a countable infinite set. 
Denote by $P^I$ the infinite tensor product w.r.t~$\phi$, and by $\pi_k : P \to P^I$ the embedding at position $k$. 
Let $\alpha : I \to I$ be any nontrivial permutation, and let $\hat\alpha : P^I \to P^I$ denote the induced automorphism given by
\begin{align*}
 \hat\alpha(\pi_{k}(x)) &= \pi_{\alpha(k)}(x), \quad \text{for } x \in P, k \in I.
\end{align*}
Then $\hat\alpha$ is (properly) outer, unless $P$ is of type $\I$ and $\{k \in I \mid \alpha(k)\not=k\}$ is finite. 
Indeed, it is well known that for any nontrivial permutation $\sigma$ on $n$ elements, the induced flip automorphism on $P \ovt \cdots \ovt P$ given by $x_1 \ot \cdots \ot x_n \mapsto x_{\sigma(1)} \ot \cdots \ot x_{\sigma(n)}$ is inner if and only if $P$ is a type $\I$ factor, see e.g.~\cite[Theorem 5]{sakai;automorphisms-tensor-products}. If $\alpha$ moves infinitely many points of $I$, the result is the content of \cite[Lemma 2.5]{vaes-verraedt;classification-type-III-bernoulli-crossed-products}. 
\end{remark}

If $(P,\phi)$ is a nontrivial factor equipped with a normal faithful state that is not almost periodic, and $\Lambda \curvearrowright I$ is any faithful action, the induced action $\Lambda \curvearrowright P^I \rtimes \R$ on the continuous core of $P^I$ is always (properly) outer, as stated in the next result. Combined with \cref{lem.only-type-III1}, we get in particular that $P^I \rtimes \Lambda$ is always a type $\III_1$ factor, even if $\Lambda$ is not icc, as its continuous core $(P^I \rtimes \Lambda) \rtimes \R = (P^I \rtimes \R) \rtimes \Lambda$ is a factor.
\begin{lemma}
\label{lem.outer-core}
Let $(P,\phi)$ be a nontrivial factor equipped with a normal faithful not almost periodic state, and let $I$ be a countable set. 
Let $\alpha : I \to I$ be any nontrivial permutation. Denote by $P^I \rtimes \R$ the crossed product with the modular action of $\phi^I$.
Then the induced automorphism $\hat\alpha \in \Aut(P^I \rtimes \R)$ on the continuous core,
given by
\begin{align*}
 \hat\alpha(\pi_{k}(x)) &= \pi_{\alpha(k)}(x), \quad \text{for } x \in P, k \in I,\\
 \hat\alpha(\lambda(t)) &= \lambda(t), \qquad\:\:\: \text{for }t \in \R,
\end{align*}
is (properly) outer.
Here, $\pi_k : P \to P^I \rtimes \R$ denotes the embedding at position $k$.
\end{lemma}
\begin{proof}
Denote $\vphi = \phi^I$, $Q = P_{\phi,\ap}$, $N = P^I \rtimes \R$, and denote by $\tilde\vphi$ the dual weight on $N$ w.r.t.~$\vphi$. 
For every finite nonempty subset $\cF \subset I$, let $K_\cF \subset L^2(N, \tilde\vphi) = L^2(\R, L^2(P^I, \vphi))$ be the subspace defined as 
\begin{align*}
 K_\cF &= L^2\big(\R, L^2( Q^I (P \ominus Q)^\cF  Q^I, \vphi)\big).
\end{align*}
By \cref{lem.almost-periodic-part-tensor-product}, we get the orthogonal decomposition
\begin{align*}
 L^2(N \ominus (P^I)_{\vphi,\ap} \rtimes \R, \tilde\vphi)
  =  L^2\big(\R, L^2(P^I \ominus (P^I)_{\vphi,\ap}, \vphi)\big)
  = \bigoplus_{\cF \in J} K_{\cF},
\end{align*}
where $J$ denotes the set of all finite nonempty subsets of $I$. Note that since $\phi$ is not almost periodic, $K_{\cF}$ is nonzero for every $\cF \in J$.

Now suppose that $\alpha : I \to I$ is a nontrivial permutation, and denote by $\hat\alpha \in \Aut(N)$ the induced automorphism. Assume that $v \in N$ is a unitary such that $vx = \hat\alpha(x)v$ for all $x \in N$. In particular, $v$ commutes with $L\R$ and hence $v \in (P^I)_\vphi \ovt L\R \subset Q^I \rtimes \R$ by \cref{lem.commutant}. Therefore, the map on $L^2(N, \tilde\vphi)$ induced by the automorphism $\Ad v$ on $N$ leaves all subspaces $K_\cF$ invariant, whereas $\hat\alpha$ sends $K_\cF$ to $K_{\alpha(\cF)}$ for all finite nonempty subsets $\cF \subset I$. This is absurd, as $\alpha$ is a nontrivial permutation. 
We conclude that $\hat\alpha$ is outer.
\end{proof}

If however $(P,\phi)$ is a nontrivial factor equipped with an almost periodic state, and $\Lambda \curvearrowright I$ is any faithful action, then the induced action $\Lambda \curvearrowright P^I \rtimes \R$ on the continuous core of $P^I$ is not always properly outer. For example, if $P$ is a type $\I$ factor and $g \in \Lambda$ moves only a finite number of points of $I$, i.e.~the set $\{ k \in I \mid g \cdot k \not= k\}$ is finite, then the induced automorphism by $g$ on $P^I \rtimes \R$ is implemented by a unitary. However, if we require that any nontrivial element moves infinitely many points of $I$, the induced action $\Lambda \curvearrowright P^I \rtimes \R$ is always properly outer, due to the following result.
\begin{lemma}
Let $(P,\phi)$ be a nontrivial factor equipped with a normal faithful almost periodic state, and let $I$ be a countable infinite set. 
Let $\alpha : I \to I$ be any nontrivial permutation.
Denote by $P^I \rtimes \R$ the crossed product with the modular action of $\phi^I$.
Then the induced automorphism $\hat\alpha \in \Aut(P^I \rtimes \R)$ on the continuous core,
given by 
\begin{align*}
 \hat\alpha(\pi_{k}(x)) &= \pi_{\alpha(k)}(x), \quad \text{for } x \in P, k \in I,\\
 \hat\alpha(\lambda(t)) &= \lambda(t), \qquad\:\:\: \text{for }t \in \R,
\end{align*}
is properly outer, unless $P$ is a type $I$ factor and the set $\{k \in I \mid \alpha(k)\not=k\}$ is finite.
Here, $\pi_k : P \to P^I \rtimes \R$ denotes the embedding at position $k$.
\label{lem.outer-core-ap}
\end{lemma}
\begin{proof}
Denote $\vphi = \phi^I$, $N = P^I \rtimes \R$, and assume that $v \in N$ is a nonzero element such that $vx = \hat\alpha(x)v$ for all $x \in N$. Since $L\R$ is fixed under $\hat\alpha$, $v$ commutes with $L\R$ and hence $v \in (P^I)_\vphi \ovt L\R$. Putting $Q = (P^I)_\vphi$, we get in particular that the restricted automorphism $\hat\alpha|_{Q \ovt L\R}$ also satisfies $vx=\hat\alpha|_{Q \ovt L\R}(x)v$ for all $x \in Q \ovt L\R$. Taking an appropriate $\omega \in (L\R)_\star$, we get a nonzero element $w = (1\ot \omega)(v) \in Q$ satisfying $wx=\hat\alpha(x)w$ for all $x \in Q$. Note that $Q$ is a factor by \cite[Lemma 2.4]{vaes-verraedt;classification-type-III-bernoulli-crossed-products}, hence $w$ is a multiple of a unitary, and by normalising, we may assume that $w$ is a unitary. 

Denote by $\Gamma$ the countable subgroup of $\Rdual$ generated by the point spectrum of $\Delta_\phi$, endowed with the discrete topology. Put $G = \hat\Gamma$ and let $\sigma : G \curvearrowright P^I$ denote the extension of $\sigma^\vphi$ to an action of the compact group $G$, as in section 2.2 of \cite{vaes-verraedt;classification-type-III-bernoulli-crossed-products}. By \cite[Lemma 2.1]{vaes-verraedt;classification-type-III-bernoulli-crossed-products}, it follows that $\hat\alpha = \Ad w \circ \sigma_s$ on $P^I$ for some $s \in G$, and by \cite[Lemma 2.5]{vaes-verraedt;classification-type-III-bernoulli-crossed-products}, it follows that the set $\{k \in I \mid \alpha(k)\not=k\}$ is then necessarily finite. 
Moreover, fixing $k \in I$ such that $\alpha(k)\not= k$, we get that $w \pi_k(P) = \pi_{\alpha(k)}(P) w$ as subalgebras of $P^I$. This is only possible if $P$ is a type $\I$ factor.

For the converse, note that if $P$ is a type $\I$ factor and $\alpha : I \to I$ is a permutation such that the set $\{k \in I \mid \alpha(k) \not=k \}$ is finite, then there exists a unitary $u \in P^I$ such that $u\pi_k(x)u^\star = \pi_{\alpha(k)}(x)$ for all $k \in I, x \in P$. Since $\vphi \circ \Ad u= \vphi$, we have $u \in (P^I)_{\vphi^I}$, and in particular it follows that $u\lambda(t)u^\star = \lambda(t)$ for all $t \in \R$. We have shown that $\hat\alpha = \Ad u$ on $P^I \rtimes \R$, thus $\hat\alpha$ is not properly outer.
\end{proof}

We saw earlier that if $(P,\phi)$ is a nontrivial factor equipped with a normal faithful state that is not almost periodic, and $\Lambda$ is a countable infinite group, then $P^\Lambda \rtimes \Lambda$ is a type $\III_1$ factor. The next remark shows that if $\Lambda$ is nonamenable, $P^\Lambda \rtimes \Lambda$ does not admit any almost periodic weight.
Recall that a von Neumann algebra $M$ is full if $\Inn(M) \subset \Aut(M)$ is closed for the topology given by 
$\alpha_n \to \alpha$ if and only if $\forall \psi \in M_\star : \|\psi \circ \alpha_n -\psi \circ \alpha\| \to 0$,
and that if $M$ is full, Connes' $\tau$-invariant of $M$ is given by the weakest topology on $\R$ making $\R \to \Out(P^\Lambda \rtimes \Lambda) : t \mapsto \sigma^{\phi^\Lambda}_t$ continuous.

\begin{remark}\label{rem.no-almost-periodic}
Let $(P,\phi)$ be a nontrivial factor equipped with a normal faithful state that is not almost periodic, and $\Lambda$ a countable infinite nonamenable group. By \cite[Lemma 2.7]{vaes-verraedt;classification-type-III-bernoulli-crossed-products}, $P^\Lambda \rtimes \Lambda$ is a full factor, and Connes' $\tau$-invariant is the weakest topology on $\R$ making $\R \to \Aut(P) : t \mapsto \sigma^\phi_t$ continuous. 
Since $\phi$ is not almost periodic, the completion of this topology cannot be compact. It follows that $P^\Lambda \rtimes \Lambda$ has no almost periodic weights, 
see the proof of \cite[Corollary 5.3]{connes;almost-periodic-states}.
\end{remark}

\section{A technical lemma}

Let $(P,\phi)$ be a factor equipped with a normal faithful state, and let $\Lambda$ be a countable group acting on an infinite set $I$, such that the action $\Lambda \curvearrowright I$ has no invariant mean. Put $\vphi = \phi^I$ and $N = P^I \rtimes_{\sigma_\vphi} \R$. Denote the action $\Lambda \curvearrowright N$, induced by the generalized Bernoulli action of $\Lambda$ on $P^I$, by $\alpha$.

The main result of this section is the following technical lemma, allowing us to locate fixed point subalgebras of $N$ w.r.t.~actions which are outer conjugate to the Bernoulli action $\alpha$. 
As in the proof of \cite[Lemma 3.3]{brothier-vaes;prescribed-fundamental-group}, locating such subalgebras allows us to deduce cocycle conjugacies between two Bernoulli actions. 
The proof of \cref{lem.pointwise-fixed-subalgebra-intertwines-in-R} follows the lines of the proof of \cite[Theorem 4.1]{popa;superrigidity-malleable-actions-spectral-gap}. 

\begin{lemma}\label{lem.pointwise-fixed-subalgebra-intertwines-in-R}
Let $p \in L\R \subset N$ be a projection with $0 < \Tr(p) < \infty$.
Assume that $(V_g)_{g \in \Lambda} \in \mathcal{U}(pNp)$ are unitaries, and that $\Omega : \Lambda \times \Lambda \to \T$ is a map such that $V_g \alpha_g(V_h) = \Omega(g,h) V_{gh}$ for all $g,h \in \Lambda$.  
Assume that $Q \subset pNp$ is a subalgebra such that for all $x \in Q$ and for all $g \in \Lambda$, $V_g \alpha_g(x) V_g^\star = x$, and assume that $q \in Q' \cap pNp$ is a projection such that for all $g \in \Lambda$, $V_g\alpha_g(q)V_g^\star \sim q$ inside $Q' \cap pNp$. Then it follows that $qQ \prec_{pNp} pL\R$.

Put $M = N \cap (L\R)' = (P^I)_\vphi \ovt L\R$.
If moreover holds that $q \in Q' \cap pMp$ and $qL\R \subset qQ \subset qMq$, then $qQ \prec_{pMp} pL\R$.
\end{lemma}
See \cref{lem.commutant} for the equality $N \cap (L\R)' = (P^I)_\vphi \ovt L\R$.

We use the spectral gap methods of \cite{popa;superrigidity-malleable-actions-spectral-gap}, applied to the following variant of Popa's malleable deformation \cite{popa;rigidity-non-commutative-bernoulli}, due to Ioana \cite{ioana;rigidity-results-wreath-product}.
Let $\Ptil = P \star L\Z$ the free product with respect to the state $\phi$ and the natural trace on $L\Z$, and denote the induced state also by $\phi$. Denote by $u_n, n \in \Z$ the canonical unitaries of $L\Z$. Let $f : \T \to (-\pi,\pi]$ be the unique function determined by $t = \exp(i f(t)), t \in \T$, and let $h =f(u_1)$ be the hermitian operator such that $u_1 = \exp(i h)$. Define $u_t = \exp(i th)$ for every $t \in \R$. 
Equip $\Ptil^I$ with the one-parameter group of state preserving automorphisms $\theta_t$ given by the infinite tensor product of $\Ad u_t$, for $t \in \R$. Define the period 2 automorphism $\gamma$ of $\Ptil^I$ as the infinite tensor product of the automorphism of $\Ptil$ satisfying $x \mapsto x$ for $x \in P$ and $u_1 \mapsto u_{-1}$.

Put $\Ntil = \Ptil^I \rtimes_{\sigma_\vphi} \R$, and denote also the action $\Lambda \curvearrowright \Ntil$, induced by the generalized Bernoulli action of $\Lambda$ on $P^I$, by $\alpha$. The automorphisms $\theta_t$ and $\gamma$ naturally extend to $\Ntil$, by acting as the identity on $L\R$. 

The condition that $\Lambda \curvearrowright I$ has no invariant mean, implies the following result, which is very similar to \cite[Lemma 3.4]{brothier-vaes;prescribed-fundamental-group}.  
\begin{lemma}\label{lem.rho-not-amenable}
Assume that $(V_g)_{g \in \Lambda}$ are unitaries in $\mathcal{U}(N)$, and $\Omega : \Lambda \times \Lambda \to \T$ is a map such that $V_g \alpha_g(V_h) = \Omega(g,h) V_{gh}$ for all $g,h \in \Lambda$. 
The unitary representation
\begin{align*}
\rho : \Lambda \to \mathcal{U}(L^2(\Ntil \ominus N, \Tr)) : \rho_g(\xi) = V_g \alpha_g(\xi) V_g^\star
\end{align*}
does not weakly contain the trivial representation.
\end{lemma}
\begin{proof}
Denote by $\tilde\vphi$ be the dual weight on $\Ntil$ and $N$ w.r.t.~the weight $\vphi$ on $\Ptil^I$ and $P^I$ respectively. 
Define for every finite subset $\cF \subset I$, the subspace $K_\cF \subset L^2(\tilde N,\tilde\vphi)$ as the $\|\fdot\|_2$-closure of the linear span of $P^I(\Ptil \ominus P)^\cF P^I \fK(\R)$, where $\fK(\R)$ denotes the set of compactly supported continuous functions on $\R$, and $\fK(\R)\subset L\R$ as convolution operators. 
Note that if $x \in P^I(\Ptil \ominus P)^\cF P^I \fK(\R), y \in P^I(\Ptil \ominus P)^\cK P^I \fK(\R)$ for $\cF\not=\cK$, then $E_{L\R}(y^\star x)=0$. In particular, we have the orthogonal decomposition
\begin{equation*}
 L^2(\Ntil \ominus N, \Tr) = \bigoplus_{\cF \in J} K_{\cF}.
\end{equation*}
Remark
that $\rho_g(K_\cF) = K_{g\cF}$ for all $g \in \Lambda$. Denote by $p_\cF$ the orthogonal projection onto $K_\cF$. 

Assume that $\rho$ does weakly contain the trivial representation, then we find an $\Ad \rho(\Lambda)$-invariant mean $\mu$ on $B(H)$.  
Define the map $\Theta : \ell^\infty(J) \to B(H)$ given by $\Theta(F) = \sum_{\cF \in J} F(\cF)p_\cF$. Since $\Theta(g \cdot F) = \rho_g \Theta(F) \rho_g^\star$, the composition $\mu \circ \Theta$ yields a $\Lambda$-invariant mean on $J$. But then $\Lambda \curvearrowright I$ admits an invariant mean, by \cite[Lemma 2.6]{vaes-verraedt;classification-type-III-bernoulli-crossed-products}. This is absurd, hence $\rho$ does not weakly contain the trivial representation.
\end{proof}

\begin{proof}[Proof of \Cref{lem.pointwise-fixed-subalgebra-intertwines-in-R}]
Let $Q \subset pNp$ a subalgebra such that for all $x \in Q$ and for all $g \in \Lambda$, $V_g \alpha_g(x)V_g^\star = x$, and let $q \in Q' \cap pNp$ be a projection with $V_g \alpha_g(q) V_g^\star \sim_{(Q' \cap pNp)} q$  for all $g \in \Lambda$.

Define as above the malleable deformation $(\theta_t)_{t \in \R}$ of $\Ntil$, and remark that $\theta_t(p) = p$ for all $t \in \R$. Let $\rho$ be the unitary representation of $\Lambda$ on $L^2(p(\Ntil \ominus N)p, \Tr)$ given by $\rho_g(\xi) = V_g \alpha_g(\xi) V_g^\star$. By \cref{lem.rho-not-amenable}, we find a constant $\kappa > 0$ and a finite subset $\cF \subset \Lambda$ such that
\begin{equation}\label{eq.use-nonamenable}
 \|\xi\|_2 \leq \kappa \sum_{g \in \cF} \|\rho_g(\xi)-\xi\|_2 \quad\text{for all }\xi \in L^2(p(\Ntil \ominus N)p, \Tr).
\end{equation}
Put $\epsilon = \|q\|_2/6$ and $\delta = \epsilon/(2 \kappa |\cF|)$. Take an integer $n_0$ large enough such that $t = 2^{-n_0}$ satisfies
\begin{align*}
\| (V_g - \theta_t(V_g))p\|_2 &\leq \delta \quad\text{for all }g \in \cF,\\
\| q - \theta_t(q)\|_2 &\leq \epsilon.
\end{align*}
For every $x \in Q$ we have $V_g \alpha_g(x) V_g^\star = x$ for all $g \in \Lambda$, and therefore
\begin{equation}\label{eq.uniform-bound}
 \|V_g \alpha_g(\theta_t(x)) V_g^\star - \theta_t(x)\|_2 \leq 2\delta \quad\text{for all }g \in \cF\text{ and all }x \in Q\text{ with }\|x\|\leq 1.
\end{equation}
Denote by $E : \Ntil \to N$ the unique trace preserving conditional expectation. Whenever $x \in Q$ with $\|x\| \leq 1$, we put $\xi = \theta_t(x) - E(\theta_t(x))$ and conclude from \eqref{eq.use-nonamenable} and \eqref{eq.uniform-bound} that
\begin{align*}
\|\xi\|_2 \leq \kappa \sum_{g \in \cF} \| \rho_g(\xi)-\xi\|_2 \leq 2\kappa |\cF| \delta = \epsilon.
\end{align*}
A direct computation shows that $(\theta_t)$ satisfies the following transversality property in the sense of \cite[Lemma 2.1]{popa;superrigidity-malleable-actions-spectral-gap}:
\begin{align*}
\| x- \theta_t(x)\|_2 \leq \sqrt{2} \|\theta_t(x) - E(\theta_t(x))\|_2 \quad\text{for all } x \in pNp.
\end{align*}
We conclude that $\|x - \theta_t(x)\|_2 \leq 2\epsilon$ for all $x \in Q$ with $\|x\|\leq 1$, hence $\|y-\theta_t(y)\|_2 \leq 3\epsilon$ for all $y \in qQ$. It follows that for all $y \in \mathcal{U}(qQ)$,
\begin{align*}
  | \Tr( y \theta_t(y^\star)) - \Tr(yy^\star) | \leq \|y\|_2 \| y - \theta_t(y)\|_2 \leq \|q\|_2 3\epsilon = \Tr(q)/2,
\end{align*}
and hence $\Tr(y \theta_t(y^\star)) \geq \Tr(q)/2$ for all $y \in \mathcal{U}(qQ)$.

Let $W \in q\Ntil \theta_t(q)$ be the unique element of minimal $\|\fdot\|_2$ in the weakly closed convex hull of $\{y \theta_t(y^\star) \mid y \in \mathcal{U}(qQ)\}$. Then $\Tr(W) \geq \Tr(q)/2$ and $xW = W \theta_t(x)$ for all $x \in Q$. In particular, $WW^\star$ commutes with $Q$.

Put now for $g \in \Lambda$, $W_g = V_g \alpha_g(W) \theta_t(V_g^\star)$. Since $Q$ is pointwise fixed under $\Ad V_g \circ \alpha_g$, we then also have $x W_g = W_g \theta_t(x)$ for all $x \in Q$. The join of the left support projections of all $W_g$, $g \in \Lambda$, is a projection $q_0 \in p \Ntil p \cap Q'$ that satisfies $q_0 = V_g \alpha_g(q_0) V_g^\star$ for all $g \in \Lambda$. By \cref{lem.rho-not-amenable}, $q_0 \in N$, and in particular, $\gamma(q_0) = q_0$. We claim that we can find some $g \in \Lambda$ such that $\theta_t(\gamma(W^\star) W_g) \not= 0$. 
To prove the claim, assume that for all $g \in \Lambda$, $\gamma(W^\star)W_g = 0$. Since the join of the left support projections of all $W_g$, $g \in \Lambda$ equals $q_0$, we get $\gamma(W^\star)q_0 = \gamma(W^\star q_0) = 0$, contradiction.

Fix now $g \in \Lambda$ such that $\gamma(W^\star)W_g \not= 0$. 
Take $u \in \mathcal{U}(Q' \cap pNp)$ such that $uq=V_g \alpha_g(q)V_g^\star u$, and put $W' \defeq \theta_t(\gamma(W^\star) W_g) \theta_{2t}(u)$.
Then $W' \in q\Ntil \theta_{2t}(q)$ is a nonzero element satisfying $x W' = W' \theta_{2t}(x)$ for all $x \in Q$. Repeating the same argument $n_0$ times, we obtain a nonzero element $W \in q \Ntil \theta_1(q)$ such that $WW^\star \in Q' \cap q \Ntil q$ and
\begin{equation}\label{eq:W}
 x W = W \theta_1(x) \quad\text{for all }x \in qQ.
\end{equation}
Observe that if $q \in (L\R)'$ and $qL\R \subset qQ$, then $W$ commutes with $L\R$.

For every $\cF \subset I$, define $N(\cF) = P^\cF \rtimes_{\sigma_\vphi} \R$. We claim that there exists a finite subset $\cF \subset I$ such that $qQ \prec_{pNp} pN(\cF)p$. Suppose the claim fails, and let $x_n \in \mathcal{U}(qQ)$ be a sequence of unitaries satisfying
\begin{align*}
\| E_{pN(\cF)p}(ax_nb^\star)\|_2 \to 0 \quad\text{for all }a,b\in pNq \text{ and all finite subsets }\cF \subset I.
\end{align*}
We claim that 
\begin{equation}\label{eq.theta1ofQ-weakly-embeds}
\|E_{pNp}(a \theta_1(x_n) b^\star)\|_2 \to 0 \quad \text{for all }a,b \in p \Ntil p. 
\end{equation}
Since the linear span of all $p N \Ptil^\cF p$, $\cF \subset I$ finite, is $\|\fdot\|_2$-dense in $p\Ntil p$, it suffices to prove \eqref{eq.theta1ofQ-weakly-embeds} for all $a,b \in pN\Ptil^\cF p$ and all finite subsets $\cF \subset I$. But for all $a,b \in pN\Ptil^\cF p$, we have
\begin{align}\label{eq:hulp}
 E_{pNp}(a \theta_1(x_n) b^\star) = E_{pNp}\big(a \theta_1(E_{pN(\cF)p}(x_n)) b^\star\big),
\end{align}
thus the conclusion follows from the choice of the sequence $(x_n)$. So \eqref{eq.theta1ofQ-weakly-embeds} is proven, and since $x_n \in \mathcal{U}(qQ)$ are unitaries, it follows in particular that 
\begin{align*}
\|E_{pNp}(WW^\star)\|_2 = \|E_{pNp}(WW^\star)q\|_2 = \|E_{pNp}(WW^\star) x_n\|_2 = \| E_{pNp}(W \theta_1(x_n) W^\star)\|_2 \to 0.
\end{align*}
As $W$ is nonzero, we reached a contradiction, and we have proven the existence of a finite subset $\cF \subset I$ such that $qQ \prec_{pNp} pN(\cF)p$.

Since $qQ \prec_{pNp} pN(\cF)p$, we find $n \in \N$, a projection $p' \in pN(\cF)p \otimes M_n(\C)$, a unital normal homomorphism 
$\psi_1 : qQ \to p' \big(pN(\cF)p \otimes M_n(\C)\big) p'$ and a partial isometry $v \in qNp \otimes M_{1,n}(\C)$ such that $v^\star v \leq p'$ and $x v = v\psi_1(x)$ for all $x \in qQ$. 
For every $g \in \Lambda$, choose $u_g \in \mathcal{U}(Q' \cap pNp)$ such that $u_g V_g \alpha_g(q)V_g^\star = q u_g$. Put $\mathcal{G} \subset \mathcal{U}(Q'\cap pNp)$ the subgroup of all unitaries $u \in Q' \cap pNp$ satisfying $uq = qu$.
Let $\cK \subset \Lambda$ and $\mathcal{V} \subset \mathcal{G}$ be finite sets such that 
\begin{align*}
\Tr\left(\bigvee_{g \in \cK, u \in \mathcal{V}} u u_g V_g \alpha_g(vv^\star)V_g^\star u_g^\star u^\star\right) > \frac12 \Tr\left(\bigvee_{g \in \Lambda, u \in \mathcal{G}} uu_g V_g \alpha_g(vv^\star)V_g^\star u_g^\star u^\star\right).
\end{align*}
Remark that for all $g \in \cK$, all $u \in \mathcal{V}$, and all $x \in qQ$, we have that
\begin{align*}
x u u_g V_g \alpha_g(v) = u u_g V_g \alpha_g(v)\, \alpha_g(\psi_1(x)).
\end{align*}
Replacing $\cF$ by $\cK\cdot\cF$, $\psi_1$ by the direct sum of the maps $\alpha_g\circ \psi_1$ for $g \in \cK$, $u \in \mathcal{V}$ and $v$ by the polar part of the direct sum of the $uu_g V_g \alpha_g(v)$'s, we obtain a unital normal $\star$-homomorphism 
$\psi_1 : qQ \to p' \big(pN(\cF)p \otimes M_n(\C)\big) p'$ and a partial isometry $v \in qNp \otimes M_{1,n}(\C)$ such that $v^\star v \leq p'$,
\begin{align}\label{eq.action-does-not-make-v-zero}
u_g V_g \alpha_g(vv^\star) V_g^\star u_g^\star \not\perp vv^\star \text{ for all }g \in G,
 \quad\text{and}\quad x v = v\psi_1(x) \text{ for all }x \in qQ. 
\end{align}
Moreover, we may assume that the support projection of $E_{p N(\cF) p \otimes M_n(\C)}(v^\star v)$ equals $p'$.

Since the action $\Lambda \curvearrowright I$ has no invariant mean, a fortiori all orbits are infinite and thus there exists some $g \in \Lambda$ such that $g \cF \cap \cF = \emptyset$ (see e.g. \cite[Lemma 2.4]{popa-vaes;strong-rigidity-generalized-bernoulli-actions}  ).
Putting $\psi_2 : qQ \to \alpha_g(p')\big(p N(g\cF)p \otimes M_n(\C)\big)\alpha_g(p')$ given by $\psi_2 = \alpha_g \circ \psi_1$, we get that
\begin{equation}\label{eq.intertwines}
 \psi_1(x)\: v^\star u_gV_g \alpha_g(v) = v^\star u_gV_g \alpha_g(v)\: \psi_2(x) \quad\text{for all }x\in qQ.
\end{equation}
By the first part of \eqref{eq.action-does-not-make-v-zero}, $v^\star u_gV_g \alpha_g(v)\not=0$.

We claim that $\psi_1(qQ) \prec_{pNp \otimes M_n(\C)} pN(\emptyset)p \otimes M_n(\C) = pL\R \otimes M_n(\C)$. Suppose that this is not the case, then we find a sequence $x_n \in qQ$ such that $\psi_1(x_n) \in \mathcal{U}(\psi_1(qQ))$ and  
\begin{align*}
 \|E_{L\R \otimes M_n(\C)}(a \psi_1(x_n) b)\|_2 \to 0 \quad\text{for all }a,b \in pNp \otimes M_n(\C).
\end{align*}
Note that for all $a,b \in pN(\cF)p \otimes M_n(\C)$ and all $a',b' \in pN(g\cF)p \otimes M_n(\C)$,
\begin{align*}
 E_{pN(g \cF)p \otimes M_n(\C)}(a' a \psi_1(x_n) b b') = a' E_{L\R \otimes M_n(\C)}\big(a \psi_1(x_n) b\big) b',
\end{align*}
and by density and \eqref{eq.intertwines}, it follows that $\|E_{p N(g\cF)p \otimes M_n(\C)}(y \psi_1(x_n)z)\|_2\to 0$ for all $y,z \in pNp \otimes M_n(\C)$.
Putting $w = v^\star u_gV_g \alpha_g(v)$, we have that 
\begin{align*}
0 \leftarrow \|E_{pN(g \cF)p \otimes M_n(\C)}( w^\star  \psi_1(x_n)  w) \|_2
    &= \|E_{pN(g \cF)p \otimes M_n(\C)}( \psi_2(x_n)\, w^\star  w) \|_2\\
    &= \|\psi_2(x_n)\,E_{pN(g \cF)p \otimes M_n(\C)}( w^\star  w) \|_2\\
    &\geq \|w\psi_2(x_n)\,E_{pN(g \cF)p \otimes M_n(\C)}( w^\star  w) \|_2\displaybreak[0]\\
    &= \|\psi_1(x_n)w\,E_{pN(g \cF)p \otimes M_n(\C)}( w^\star  w) \|_2\\
    &= \|w\,E_{pN(g \cF)p \otimes M_n(\C)}( w^\star  w) \|_2\\
    &\geq \|w^\star w \,E_{pN(g \cF)p \otimes M_n(\C)}(w^\star w )\|_2.
\end{align*}
But this is absurd, as $w^\star w$ is nonzero. 
We conclude that $\psi_1(qQ) \prec pL\R \otimes M_n(\C)$ inside $pNp \otimes M_n(\C)$. Since the support of $E_{p N(\cF) p \otimes M_n(\C)}(v^\star v)$ equals $p'$, this intertwining can be combined with the intertwining given by $v$, and we obtain that $qQ \prec_{pNp} pL\R$.

Assume now that $q \in Q' \cap (L\R)' \cap N$ and $qL\R \subset qQ \subset qMq$, where $M = N \cap (L\R)' = (P^I)_\vphi \ovt L\R$.
By \eqref{eq:W}, 
we obtain a nonzero element $W \in q \Ntil \theta_1(q) \cap (L\R)'$
such that  $x W = W \theta_1(x)$ for all $x \in qQ$.

For every $\cF \subset I$, define $M(\cF) = M \cap N(\cF) = (P^\cF)_\vphi \ovt L\R$. We claim that there exists a finite subset $\cF \subset I$ such that $qQ \prec_{pMp} pM(\cF)p$. If the claim fails, then there exists a sequence of unitaries $x_n \in \mathcal{U}(qQ)$ such that 
\begin{align*}
\| E_{pM(\cF)p}(ax_nb^\star)\|_2 \to 0 \quad\text{for all }a,b\in pMq \text{ and all finite subsets }\cF \subset I.
\end{align*}
By \eqref{eq:hulp} and using that $E_{pM(\cF)p}(x_n) = E_{pN(\cF)p}(x_n)$, it follows that $\|E_{pNp}(a \theta_1(x_n) b^\star)\|_2 \to 0$ for all $a,b \in p \Ntil p \cap (pL\R)'$,
and thus $\|E_{pNp}(WW^\star)\|_2 \to 0$, contradiction.
Hence the claim is proven.

Replacing $N$ by $M$ and $N(\cF)$ by $M(\cF)$ in the last part of the proof of the general case, we then obtain that $qQ \prec_{pMp} pL\R$. We need $qL\R \subset qQ$ to establish that $uu_gV_g \alpha_g(q) \in (L\R)'$ for all $u \in \mathcal{G}, g \in \Lambda$.
\end{proof}

Using \cref{lem.pointwise-fixed-subalgebra-intertwines-in-R}, we obtain the following result. 
Note that this result follows immediately from \cite[Theorem A.1]{popa;class-factors-betti-number-invariants} if both centralizers $(P_i^{I_i})_{\phi_i^{I_i}}$ are trivial, since then $L\R \subset P_i^{I_i} \rtimes \R$ is a maximal abelian subalgebra.
\begin{lemma}
\label{lem.pointwise-fixed-subalgebra-intertwines-in-R-conclusion}
Let $(P_0,\phi_0)$ and $(P_1,\phi_1)$ be nontrivial factors equipped with normal faithful nonperiodic states. Let $\Lambda$ be a countable group acting on infinite sets $I_0$, $I_1$, such that the actions $\Lambda \curvearrowright I_i$ do not admit invariant means. Denote by $\vphi_i = \phi_i^{I_i}$ the canonical state on $P_i^{I_i}$. 

Assume that the centralizers $(P_i^{I_i})_{\vphi_i}$ are factors.
Assume that $\psi : P_0^{I_0} \rtimes_{\sigma^{\vphi_0}} \R \to P_1^{I_1} \rtimes_{\sigma^{\vphi_1}}\R$ is an isomorphism such that the induced actions $\Lambda \curvearrowright P_i^{I_i} \rtimes_{\sigma^{\vphi_i}} \R$ are cocycle conjugate through $\psi$. Then there exists a partial isometry $w \in P_1^{I_1} \rtimes \R$ such that $w w^\star \in \psi(L\R)'$, $w^\star w \in (L\R)'$ and $\psi(L\R)w = wL\R$.
\end{lemma}
\begin{proof}
Let $\psi$ be an isomorphism as in the statement.
Put $N_i = P_i^{I_i}\rtimes \R$ and denote the action of $\Lambda$ on $N_i$ by $\alpha^i$. 
Identify $\Rdual$ as the dual of $\R$ using the pairing $\inp{t}{\mu} = \mu^{\mathbf{i}{t}}$ for $t\in\R,\mu\in\Rdual$, and denote by $\hat\sigma^{\vphi_i} : \R^+_0 \curvearrowright N_i$ the dual action w.r.t.~$\sigma^{\vphi_i}$. Note that $\hat\sigma^{\vphi_0}_\mu(L\R) = L\R$ for all $\mu \in \Rdual$ and that $\hat\sigma^{\vphi_0}$ commutes with the action of $\Lambda$ on $N_0$, hence we may replace $\psi$ by $\psi \circ \hat\sigma^{\vphi_0}_\mu$ for an appropiate $\mu \in\Rdual$ and assume that $\psi$ is a trace preserving isomorphism implementing the cocycle conjugation. This will allow us to ease the notations.

{\bf Step 1.} 
First, we prove that for every nonzero finite trace projection $q \in L\R$, there exists a nonzero partial isometry $w \in \psi(q) N_1 q$ such that $ww^\star \in \psi(qL\R)'$, $w^\star w \in (qL\R)'$ and $\psi(L\R)w \subset wL\R$. 

Note that $N_1$ is a factor by \cref{lem.only-type-III1}, since $\phi_1$ is nonperiodic.
Let $q \in L\R$ be any nonzero finite trace projection, and take $u \in \mathcal{U}(N_1)$ such that $u \psi(q) u^\star = q$. Then $\Ad u \circ \psi$ also defines a cocycle conjugation between the actions $\Lambda \curvearrowright qN_iq$, hence we find a 1-cocycle $V_g, g \in \Lambda$ for $\alpha^1$ such that 
\begin{align*}
(\Ad u \circ \psi) \circ \alpha^0_g = \Ad V_g \circ \alpha^1_g \circ (\Ad u \circ \psi), \quad\text{for all }g\in \Lambda.
\end{align*}
Denote $\psi_u = \Ad u \circ \psi$. By construction, $\psi_u(q L\R)$ is an abelian subalgebra of $q N_1 q$, such that for all $x \in \psi_u(q L\R)$ and for all $g \in \Lambda$, we have that $V_g \alpha_g^1(x)V_g^\star = x$.
By \cref{lem.pointwise-fixed-subalgebra-intertwines-in-R-conclusion}, $\psi_u(q L\R) \prec_{qN_1q} qL\R$. 
  
Note that the relative commutant of $q L\R$ inside $q N_i q$ is given by $(qL\R)' \cap qN_iq = q((L\R)' \cap N_i)q = (P_i^{I_i})_{\vphi_i} \ovt q L\R$, see \cref{lem.commutant}.
Put $Q_i =  (P_i^{I_i})_{\vphi_i} \ovt q L\R$, and observe that $Q_i' \cap qN_iq= q L\R$ by the assumption that $(P_i^{I_i})_{\vphi_i}$ is a factor.
Since $\psi_u(q L\R) \prec_{qN_1q} qL\R$, 
it follows from \cref{lemma:intertwining-abelian-subalg} that there exists a partial isometry $w_0 \in qN_1q$ such that $w_0w_0^\star \in \psi_u(qL\R)'$, $w_0^\star w_0 \in (qL\R)'$ and $\psi_u(qL\R) w_0 \subset w_0 (qL\R)$. Then $w = u^\star w_0$ is the desired partial isometry.
      
{\bf Step 2.} 
Fix now a nonzero finite trace projection $q \in L\R$, and let $w \in \psi(q)N_1q$ be a nonzero partial isometry such that $ww^\star \in \psi(qL\R)'$, $w^\star w \in (qL\R)'$ and $\psi(L\R)w \subset wL\R$.
Put $p_0 = \psi^{-1}(ww^\star)$, $p_1 = w^\star w$, and put $M_i = p_i(P_i^{I_i} \rtimes \R) p_i$. Denote by $\theta : M_0 \to M_1$ the isomorphism given by $\theta = \Ad w^\star \circ \psi$, and note that $\theta(p_0 L\R) \subset p_1 L\R$, hence, taking the relative commutants, 
\begin{align}\label{eq:inclusions}
p_0 L\R \subset \theta^{-1}(p_1L\R) \subset p_0( (P_0^{I_0})_{\vphi_0} \ovt qL\R)p_0.
\end{align}
Let $u \in \mathcal{U}(\psi(q)N_1q)$ be a unitary such that $u p_1 = w$, and put $Q = (\Ad u^\star \circ \psi)^{-1}(qL\R)$.
Since $\Ad u^\star \circ \psi$ is a cocycle conjugation between the actions $\Lambda \curvearrowright qN_iq$, we find a 1-cocycle $V_g, g \in \Lambda$ for $\alpha^0$ such that 
\begin{align*}
\alpha_g^1 \circ (\Ad u^\star \circ \psi) = (\Ad u^\star \circ \psi) \circ \Ad V_g \circ \alpha_g^0. 
\end{align*}
In particular, for all $x \in Q$ we have 
we have $x = V_g \alpha^0_g(x)V_g^\star$. Furthermore, \eqref{eq:inclusions} means that $p_0 L\R \subset p_0 Q \subset p_0 (qN_0q \cap (L\R)') p_0$.
Also note that for every $g \in \Lambda$, the projections $p_1$ and $\alpha_g^0(p_1)$ both belong to $(P_1)_{\vphi_1} \ovt qL\R$, and the central traces of these projections inside $(P_1)_{\vphi_1} \ovt qL\R$ coincide. Indeed, $(P_1)_{\vphi_1}$ is a factor and the central trace is therefore given by the conditional expectation $E : (P_1)_{\vphi_1} \ovt qL\R \to qL\R$, which satisfies $E = E \circ \alpha_g^0$. In particular, the projections $p_1$ and $\alpha_g^0$ are equivalent, and so are $p_0$, $V_g\alpha_g(p_0)V_g^\star$ inside $Q'\cap qN_0q$.

By \cref{lem.pointwise-fixed-subalgebra-intertwines-in-R}, we deduce that $\theta^{-1}(p_1L\R) \prec_{(P_0)_{\vphi_0} \ovt qL\R} q L\R$. Since $\theta^{-1}(p_1L\R)$ and $qL\R$ are abelian, we then find a partial isometry $v \in (P_0)_{\vphi_0} \ovt qL\R$ with $vv^\star \leq p_0$, such that $\theta^{-1}(p_1L\R)v \subset vL\R$. But note that $v$ commutes with $L\R$, hence 
\begin{align*}
v^\star v L\R = v^\star (p_0L\R)v \subset v^\star \theta^{-1}(p_1L\R) v \subset v^\star v L\R.
\end{align*}
Putting $q_0 = \theta(vv^\star) \leq p_1$, it follows that $q_0 \in \theta(p_0L\R)'$, $q_0 \in (p_1L\R)'$ and $q_0 \theta(p_0L\R) = q_0 (p_1 L\R)$.
Then $w_0 = wq_0 \in \psi(q)N_1 q$ is a nonzero partial isometry such that $w_0 w_0^\star = wq_0w^\star = \psi(vv^\star) \in \psi(qL\R)'$, $w_0^\star w_0 = q_0 \in (qL\R)'$, and moreover
\begin{align*}
w_0^\star w_0 (L\R)
=q_0 (p_1 L\R) q_0
= q_0 \theta(p_0 L\R) q_0
= w_0^\star \psi(L\R)w_0.
\end{align*}
We conclude that $\psi(L\R)w_0 = w_0 L\R$, which completes the proof.
\end{proof}

\section{A non-isomorphism result for generalized Bernoulli crossed products}

Recall that for any factor $(P,\phi)$ equipped with a normal faithful state, $P_{\phi,\ap} \subset P$ denotes the subalgebra spanned by the eigenvectors of $\Delta_\phi$, i.e.~
\begin{align*}
P_{\phi,\ap} = \big(\lspan \bigcup_{\mu \in \Rdual} P_{\phi,\mu} \big)''. 
\end{align*}
The following theorem provides a non-isomorphism result for all generalized Bernoulli crossed products, with amenable factors $(P,\phi)$ as base algebras, for which the almost periodic part $P_{\phi,\ap}$ is again a factor.  
In particular, it will provide a proof of \cref{thmstar.A}.

\begin{theorem}\label{thm.non-isomorphism-all-generalized-bernoulli}
Let $(P_0,\phi_0)$ and $(P_1,\phi_1)$ be nontrivial amenable factors equipped with normal faithful states, such that $(P_0)_{\phi_0,\ap}$ and $(P_1)_{\phi_1,\ap}$ are factors. Let $\Lambda_0$ and $\Lambda_1$ be icc groups in the class $\cC$ that act on infinite sets $I_0$ and $I_1$ respectively.
Assume for $i=0,1$ that the action $\Lambda_i \curvearrowright I_i$ has no invariant mean, and that for every $g \in \Lambda_i - \{e\}$, the set $\{k \in I_i \mid g \cdot k \not= k \}$ is infinite.

The algebras $P_0^{I_0} \rtimes \Lambda_0$ and $P_1^{I_1} \rtimes \Lambda_1$ are isomorphic if and only if one of the following statements holds.
\begin{enumerate}[\upshape (a),topsep=0mm]
\item The states $\phi_0$ and $\phi_1$ are both tracial, and the actions $\Lambda_i \curvearrowright (P_i, \phi_i)^{\Lambda_i}$ are cocycle conjugate, modulo a group isomorphism $\Lambda_0 \cong \Lambda_1$.
\item The states $\phi_0$ and $\phi_1$ are both nontracial, and the actions $\Lambda_i \curvearrowright (P_i, \phi_i)^{\Lambda_i}$ are, up to reductions and modulo a group isomorphism $\Lambda_0 \cong \Lambda_1$, cocycle conjugate through a state preserving isomorphism.
\end{enumerate}
\end{theorem}

\begin{proof}
If (a) or (b) holds, then it is clear that the crossed products $P_i^{I_i} \rtimes \Lambda_i$ are isomorphic. For the reverse implication, let $\psi : P_0^{I_0} \rtimes \Lambda_0 \to P_1^{I_1} \rtimes \Lambda_1$ be a $\star$-isomorphism. Denote by $\vphi_i = \phi_i^{I_i}$ be the product state on $P_i^{I_i}$.
By \cref{lem.only-type-III1}, we may distinguish the following three distinct cases.

{\bf Case 1.}~{\em One of the states $\phi_i$ is tracial.}\\
If one of the states $\phi_i$ is tracial, $P_{0}^{I_{0}}$, $P_{0}^{I_{0}} \rtimes \Lambda_0$, $P_{1}^{I_{1}}$ and $P_{1}^{I_{1}} \rtimes \Lambda_1$ are all necessarily of type $\II_1$. 

Since the groups $\Lambda_i$ are icc and belong to the class $\cC$, it follows from \cite[Lemma 8.4]{ioana-peterson-popa;amalgamated-free-product} (see also \cite[lemma 4.1]{vaes-verraedt;classification-type-III-bernoulli-crossed-products}) that $\psi(P_0^{I_0})$ and $P_1^{I_1}$ are unitarily conjugate inside $P_1^{I_1} \rtimes \Lambda_1$. Since the actions $\Lambda_i \curvearrowright P_i^{I_i}$ are outer, this means that $\Lambda_0 \cong \Lambda_1$ and that the actions $\Lambda_i \curvearrowright P_i^{I_i}$ are cocycle conjugate.

{\bf Case 2.}~{\em One of the states $\phi_i$ is periodic.}\\
Note that under the extra assumption that the $P_i^{I_i} \rtimes \Lambda_i$ are full factors, the result directly follows from the proof of \cite[Theorem 6.1]{vaes-verraedt;classification-type-III-bernoulli-crossed-products}. 
We now provide a proof of the general situation. First note that since $\Lambda_i$ is icc and $\mathcal{Z}(P_i^{I_i} \rtimes \R) \subset L\R$ by \cref{lem.only-type-III1}, we get that $\mathcal{Z}((P_i^{I_i} \rtimes \R) \rtimes \Lambda_i) = \mathcal{Z}(P_i^{I_i}\rtimes \R)$, and hence $P_i^{I_i} \rtimes \Lambda_i$ and $P_i^{I_i}$ are of the same type. In particular, it follows from the type classification (see \cref{lem.only-type-III1}) that since one of the states $\phi_i$ is periodic, also the other is, and that the periods of both states $\phi_i$ are equal. Let $T >0 $ denote this period.

Put $G = \R/T\Z$, let $\sigma^i : G \curvearrowright P_i^{I_i}$ denote the actions induced by $\sigma^{\vphi_i}$, and put $N_i = P_i^{I_i} \rtimes G$. By Connes' Radon-Nykodym cocycle theorem for modular automorphism groups, $\psi$ is a cocycle conjugacy between the modular actions $\sigma^i$ of $G$ and therefore extends to a $\star$-isomorphism $\Psi : M_0 \to M_1$ between the crossed products $M_i = (P_i^{I_i} \rtimes \Lambda_i) \rtimes G = N_i \rtimes \Lambda_i$.
Note that $P_i^{I_i}$ has a factorial discrete decomposition \cite[Lemma 2.4]{vaes-verraedt;classification-type-III-bernoulli-crossed-products}, hence $N_i$ is the hyperfinite $\II_\infty$ factor. Also note that since $\Lambda_i$ has trivial center, the action $\Lambda_i \curvearrowright N_i$ is outer \cite[Lemmas 2.2 and 2.5]{vaes-verraedt;classification-type-III-bernoulli-crossed-products}. 
Proceeding exactly as in the proof of \cite[Theorem 5.1]{vaes-verraedt;classification-type-III-bernoulli-crossed-products}, we obtain that the action $\Lambda_0 \curvearrowright (P_0^{I_0}, \vphi_0)$ is cocycle conjugate to the reduced cocycle action $(\Lambda_1 \curvearrowright (P_1^{I_1}, \vphi_1))^p$ for some projection $p \in (P_1^{I_1})_{\vphi_1}$.

{\bf Case 3.}~{\em Both states $\phi_i$ are nonperiodic.}\\
Let $N_i = P_i^{I_i} \rtimes \R$ denote the crossed product of $P_i^{I_i}$ with the modular action of $\vphi_i$. Since $\phi_i$ is not periodic, $P_i^{I_i}$ is a type $\III_1$ factor by \cref{lem.only-type-III1}, and hence $N_i$ is a factor. 
It follows from either \cref{lem.outer-core} or from \cref{lem.outer-core-ap} that the induced action $\Lambda_i \curvearrowright N_i$ is outer.

By Connes' Radon-Nykodym cocycle theorem for modular automorphism groups, $\psi$ is a cocycle conjugacy between the modular automorphism groups $(\sigma^{\vphi_0}_t)_{t \in \R}$ and $(\sigma^{\vphi_1}_t)_{t \in \R}$ and therefore extends to a $\star$-isomorphism $\Psi : M_0 \to M_1$ between the crossed products $M_i = (P_i^{I_i} \rtimes \Lambda_i) \rtimes \R = N_i \rtimes \Lambda_i$.
Since the actions $\Lambda_i \actson P_i^{I_i}$ are state preserving, the action of $\Lambda_i$ on $N_i$ equals the identity on $L\R \subset N_i$.

We claim that $\Psi(N_0)$ and $N_1$ are unitarily conjugate inside $M_1$. Take a projection $p_0 \in L\R$ of finite trace. Then $\Psi(p_0)$ is a projection of finite trace in the $\II_\infty$ factor $N_1 \rtimes \Lambda_1$. After a unitary conjugacy of $\Psi$, we may assume that $\Psi(p_0) = p_1 \in L\R$. Since the projections $p_i$ are $\Lambda_i$-invariant, we have
\begin{align*}
p_i M_i p_i = p_i N_i p_i \rtimes \Lambda_i \; .
\end{align*}
The restriction of $\Psi$ to $p_0 M_0 p_0$ thus yields a $\star$-isomorphism of $p_0 N_0 p_0 \rtimes \Lambda_0$ onto $p_1 N_1 p_1 \rtimes \Lambda_1$. Because the groups $\Lambda_i$ are icc and belong to the class $\cC$, it follows from \cite[Lemma 8.4]{ioana-peterson-popa;amalgamated-free-product} that $\Psi(p_0 N_0 p_0)$ is unitarily conjugate to $p_1N_1p_1$. Since the $N_i$ are $\II_\infty$ factors, the claim follows.

By the claim in the previous paragraph, we can choose a unitary $u \in M_1$ such that $u \Psi(N_0)u^\star = N_1$. In particular, $\Lambda_0 \cong \Lambda_1$ and $\Ad u \circ \Psi$ is a cocycle conjugacy between $\Lambda_0 \actson N_0$ and $\Lambda_1 \actson N_1$. 

Identify $\Rdual$ as the dual of $\R$ using the pairing $\inp{t}{\mu} = \mu^{\mathbf{i}t}$ for $t \in \R,\mu \in \Rdual$, and denote by $\hat{\sigma}^{\vphi_i} : \Rdual \curvearrowright N_i \rtimes \Lambda_i$ the dual action w.r.t.~$\sigma^{\vphi_i}$.
We will now extend the obtained cocycle conjugacy to a cocycle conjugacy between the actions $\Lambda_i \times \Rdual \curvearrowright N_i$. 
By construction, $\Psi \circ \hat{\sigma}^{\vphi_0}_\mu = \hat{\sigma}^{\vphi_1}_\mu \circ \Psi$ for all $\mu \in \Rdual$. Therefore, $\Psi$ further extends to a $\star$-isomorphism $\tilde\Psi : M_0 \rtimes \Rdual \to M_1 \rtimes \Rdual$ satisfying $\tilde\Psi(\lambda(\mu)) = \lambda(\mu)$ for all $\mu \in \Rdual$. 
Note that we can view $M_i \rtimes \Rdual$ as $N_i \rtimes (\Lambda_i \times \Rdual)$. Putting $\Theta = \Ad u \circ \tilde\Psi$, we get an isomorphism $N_0 \rtimes (\Lambda_0 \times \Rdual)\to N_1\rtimes (\Lambda_1 \times \Rdual)$ satisfying
\begin{align*}
\Theta(N_0) &= N_1, & \Theta(N_0 \rtimes \Lambda_0) &= N_1 \rtimes \Lambda_1, &  \Theta(\lambda(\mu)) = u\hat{\sigma}^{\vphi_1}_\mu(u^\star) \lambda(\mu) \quad\text{for }\mu \in \Rdual.
\end{align*}
Using that the actions $\Lambda_i \curvearrowright N_i$ are outer, that elements in $N_i \rtimes \Lambda_i$ have a unique Fourier decomposition, and that $u\hat{\sigma}^{\vphi_1}_\mu(u^\star) \in \mathcal{N}_{N_1\rtimes \Lambda_1}(N_1)$, this implies that the restriction of $\Theta$ to $N_0$ is a cocycle conjugacy between the actions $\Lambda_i \times \Rdual \curvearrowright N_i$, modulo a continuous group homomorphism $\delta : \Lambda_0 \times \Rdual \to \Lambda_1 \times \Rdual$ satisfying $\delta(\Lambda_0) = \Lambda_1$ and $\delta(e,\mu) \in \Lambda_1 \times \{\mu\}$. Since $\Lambda_i$ has trivial center, this means that $\delta(g,\mu) = (\delta_0(g),\mu)$ for all $g \in \Lambda_0, \mu \in \Rdual$, and a group isomorphism $\delta_0 : \Lambda_0 \to \Lambda_1$.

Denoting by $\alpha_i : \Lambda_i \curvearrowright I_i$ the given actions, and by $\hat\alpha_i : \Lambda_i \to \Aut(P_i^{I_i}\rtimes \R)$ the induced actions on the continuous cores, we now have obtained that the actions $\hat\alpha_0 : \Lambda_0 \curvearrowright N_0$ and $\hat\alpha_1 \circ \delta_0 : \Lambda_0 \curvearrowright N_1$ are cocycle conjugate. 

Note that for $i=0,1$, the centralizer of $P_i^{I_i}$ w.r.t.~$\phi_i$ is a factor, by \cref{lem.technical-condition}. By \cref{lem.pointwise-fixed-subalgebra-intertwines-in-R-conclusion} with $\Lambda = \Lambda_0$ acting on $I_0$ by $\alpha_0$, and on $I_1$ by $\alpha_1 \circ \delta_0$, we find a partial isometry $w \in N_1$ such that $ww^\star \in \Theta(L\R)'$ and $w^\star w \in (L\R)'$, satisfying $\Theta(L\R)w = wL\R$. 
By \cref{lem.state-preserving-actions} below, we conclude that a reduction of the action $\alpha^0 : \Lambda_0\curvearrowright P_0^{I_0}$ is cocycle conjugate to a reduction of $\alpha^1 \circ \delta_0: \Lambda_0\curvearrowright P_1^{I_i}$, through a state preserving isomorphism. 
\end{proof}

\begin{lemma}\label{lem.state-preserving-actions}
Let $(P_0,\vphi_0)$, $(P_1,\vphi_1)$ be von Neumann algebras with normal faithful states and separable preduals. 
Let $\Lambda$ be a countable group, with state preserving actions $\Lambda \curvearrowright^{\alpha^i} (P_i,\vphi_i)$, such that the centralizers $(P_i)_{\vphi_i}$ are factors. Denote by $P_i \rtimes \R$ the crossed product of $P_i$ with the modular action of $\vphi_i$. Let $\Lambda \curvearrowright^{\tilde \alpha^i} P_i \rtimes \R$ denote the induced action given by $\tilde \alpha^i_s (\pi_{\sigma^{\vphi_i}}(x)) = \pi_{\sigma^{\vphi_i}}(\alpha^i_s(x))$ for $x \in P_i, s \in \Lambda$ and $\tilde \alpha^i_s(\lambda(t)) = \lambda(t)$ for $t \in \R$, and let $\Rdual \curvearrowright P_i \rtimes \R$ denote the dual action w.r.t.~$\sigma^{\vphi_i}$.

Assume that $\psi : P_0 \rtimes \R \to P_1 \rtimes \R$ is an isomorphism such that the induced actions $\Lambda \times \Rdual \curvearrowright P_i \rtimes \R$ are cocycle conjugate through $\psi$, and that $w \in P_1 \rtimes \R$ is a partial isometry such that $ww^\star \in \psi(L\R)'$, $w^\star w \in (L\R)'$ and $\psi(L\R)w = wL\R$. 
Then the actions $\Lambda \curvearrowright (P_i, \vphi_i)$ are, up to reductions, cocycle conjugate through a state preserving isomorphism.
\end{lemma}
\begin{proof}
Put $N_i = P_i \rtimes \R$, and denote by $\psi : N_0 \to N_1$ the given isomorphism and by $w\in N_1$ the partial isometry such that $ww^\star \in \psi(L\R)'$, $w^\star w \in (L\R)'$ and $\psi(L\R)w = wL\R$. 
Let $V_{g,\mu} \in \mathcal{U}(N_1)$ be a 1-cocycle for the action $\Lambda \times \Rdual \curvearrowright N_1$ such that
\begin{align}\label{eq:coco}
 \psi \circ \alpha_g^0 \circ \hat{\sigma}^{\vphi_0}_\mu = \Ad V_{g, \mu} \circ \alpha_{g}^1 \circ \hat{\sigma}^{\vphi_1}_\mu \circ \psi \qquad \text{for all }g \in \Lambda, \mu \in \Rdual.
\end{align}
Identify $\Rdual$ as the dual of $\R$ using the pairing $\inp{t}{\mu} = \mu^{\mathbf{i}t}$ for $t \in \R,\mu \in \Rdual$, and denote by $\hat{\sigma}^{\vphi_i} : \Rdual \curvearrowright N_i$ the dual action w.r.t.~$\sigma^{\vphi_i}$. 
Extend $\psi$ to an isomorphism $\Psi : N_0 \rtimes \R^+_0 \to N_1 \rtimes \Rdual$ by putting $\Psi(x) = \psi(x)$ for $x \in P_0 \rtimes \R$ and $\Psi(\lambda(\mu)) = V_{1,\mu} \lambda(\mu)$.
Put $\kappa \in \Rdual$ such that $\Tr_{\vphi_1} \circ \psi = \kappa \Tr_{\vphi_0}$, then $\Psi$ scales the dual weights of $\Tr_{\vphi_i}$ by the same factor $\kappa$, i.e.~$\tilde \Tr_{\vphi_1} \circ \tilde\Psi = \kappa \tilde \Tr_{\vphi_0}$.
Identify $N_i \rtimes \Rdual$ with $P_i \ovt B(L^2(\Rdual))$ through the isomorphism
\begin{align*}
\Phi_i^\mathfrak{F} :  N_i \rtimes \Rdual \to P_i \ovt B(L^2(\Rdual)),\quad \Phi_i^\mathfrak{F} = \Ad\: (1 \otimes \mathfrak{F}) \circ \Phi_i,
\end{align*}
where $\Phi_i :  N_i \rtimes \Rdual \to P_i \ovt B(L^2(\R))$ denotes the Takesaki duality isomorphism as defined in \cite[Theorem X.2.3]{takesaki;theory-operator-algebras-II}, and $\mathfrak{F} : L^2(\R) \to L^2(\Rdual)$ is the Fourier transform given by
\begin{align*}
\mathfrak{F}(f)(\mu) = \frac1{\sqrt{2\pi}} \int_{t \in \R} f(t) \mu^{-\mathbf{i}t} dt, \quad f \in L^1(\R) \cap L^2(\R).
\end{align*}
Note that the duality isomorphism $\Phi_i$ maps $\lambda(t) \in L\R$ to $\lambda(t) \in B(L^2(\R))$, and hence $\Phi_i^\mathfrak{F}(L\R) = \Ad \mathfrak{F}\, (L\R) = L^\infty(\Rdual)$. Also, $x \in (P_i)_{\vphi_i}$ is mapped to $\Phi_i^\mathfrak{F}(x) = x \otimes 1$.
Let $M$ be the operator affiliated with $L^\infty(\Rdual)$ given by $M(f)(t) = t f(t)$, and let $\omega$ be the weight on $B(L^2(\Rdual))$ given by $\omega = \Tr(M \cdot)$.
Observe that $\tilde\Tr_{\vphi_i} = (\vphi_i \otimes \omega) \circ \Phi_i^\mathfrak{F}$.
Putting $\Theta = \Phi_1^\mathfrak{F} \circ \Psi \circ (\Phi_0^\mathfrak{F})^{-1}$ and $\hat{w} = \Phi_1^\mathfrak{F}(w)$, we get an isomorphism
\begin{align*}
 \Theta : P_0 \ovt B(L^2(\Rdual)) \to P_1 \ovt B(L^2(\Rdual))
\end{align*}
satisfying
\begin{align*}
(\vphi_1 \otimes \omega) \circ \Theta &= \kappa (\vphi_0 \otimes \omega),\\
\Theta \circ (\alpha^0_g \otimes \id) &= \Ad \Phi_1^\mathfrak{F}(V_{g,1})  \circ (\alpha^1_g \otimes \id) \circ \Theta \quad\text{for all }g \in \Lambda,\\
\Theta(L^\infty(\Rdual))\,\hat{w} &=\hat{w}\, L^\infty(\Rdual).
\end{align*}
Here, the second equality follows from \eqref{eq:coco}. Put $p_0 = \Theta^{-1}(\hat{w}\hat{w}^\star)$, $p_1 = \hat{w}^\star \hat{w}$.
Taking the relative commutants of the last equality, we obtain that $\Theta(p_0 (P_0\ovt L^\infty(\Rdual))p_0)\, \hat{w} = \hat{w}\, (p_1 (P_1 \ovt L^\infty(\Rdual)) p_1)$. Denote by $M_i = p_i(P_i \ovt L^\infty(\Rdual))p_i$ and let $\theta : M_0 \to M_1$ be the isomorphism given by $\theta = \Ad \hat{w}^\star \circ \Theta$. 

For $i=0,1$ and $g \in \Lambda$, note that the projections $p_i$ and $\alpha_g^i(p_i)$ both belong to $(P_i)_{\vphi_i} \ovt L^\infty(\Rdual)$, since 
\begin{align*}
p_i \in \Psi_i^{\mathfrak{F}}(N_i \cap (L\R)') = \Psi_i^{\mathfrak{F}}((P_i)_{\vphi_i} \ovt L\R) = (P_i)_{\vphi_i} \ovt L^\infty(\Rdual).
\end{align*}
Moreover, the central traces of these projections inside $(P_i)_{\vphi_i} \ovt L^\infty(\Rdual)$ coincide,
since $(P_i)_{\vphi_i}$ is a factor and the central trace is thus given by the conditional expectation $E : (P_i)_{\vphi_i} \ovt L^\infty(\Rdual) \to L^\infty(\Rdual)$, which satisfies $E = E \circ \alpha_g^i$.

Therefore, we find a partial isometry $v_{i,g} \in (P_i)_{\vphi_i} \ovt L^\infty(\Rdual)$ such that $p_i = v_{i,g}v_{i,g}^\star$ and $\alpha_g(p_i) = v_{i,g}^\star v_{i,g}$. Also choose $v_{i,e} = p_i$. Then the formula $\alpha_g^{p_i}(pxp) = v_{i,g}(\alpha_g^i \otimes \id)(pxp)v_{i,g}^\star$ for $x \in M_i$ defines a cocycle action $\alpha^{p_i} : \Lambda \curvearrowright (p_iM_ip_i, \vphi_i^p)$, where $\vphi_i^{p_i}$ is the state given by $\vphi_i^{p_i}(pxp) = \vphi_i(pxp)/\vphi(p_i)$ for $x \in M_i$.
Putting $W_g =\hat{w}^\star \Theta(v_{0,g}) \Phi_1^\mathfrak{F}(V_{g,1}) \alpha_g^1(\hat{w}) v_{1,g}^\star \in p_1\Phi^\mathfrak{F}_1(N_1)p_1 \subset p_1(P_1 \ovt B(L^2\Rdual))p_1$, we now obtain that, as isomorphisms $M_0 \to M_1$,
\begin{align}\label{eq:conj-theta}
\theta \circ \alpha_g^{p_0} = \Ad W_g \circ \alpha_g^{p_1} \circ \theta \quad\text{for all }g\in \Lambda.
\end{align}
Note that by construction, $\alpha_g^{p_i}$ is the trivial action on $p_iL^\infty(\Rdual)$, and $\theta(p_0L^\infty(\Rdual)) = p_1L^\infty(\Rdual)$. Thus we have that $W_g \in p_1\Phi_1^\mathfrak{F}(N_1)p_1 \cap (p_1L^\infty(\Rdual))' = p_1( (P_1)_{\vphi_1} \ovt L^\infty(\Rdual))p_1$.

For $i = 0, 1$, consider the integral decomposition 
\begin{align*}
 L^2(P_i,\vphi_i) \ovt L^2(\Rdual) = \int_{\Rdual}^\oplus L^2(P_i,\vphi_i) ds, \quad (P_i \ovt L^\infty(\Rdual), \vphi_i \otimes \omega) = \int_{\Rdual}^\oplus (P_i, \vphi_i)\, ds.
\end{align*}
Note that $ds$ satisfies $\int_{\mu_1}^{\mu_2} ds = \int_{\log \mu_1}^{\log \mu_2} e^t d\lambda(t)$, with $d\lambda$ the Lebesgue measure on $\R$.
In this disintegration, we can write $p_i = \int_\Rdual^\oplus p_i(s) ds$ for a measurable field of projections $s \mapsto p_i(s)$ in $P_i$. Then we find as disintegration of $M_i$,
\begin{align*}
 M_i = p_i(P_i \ovt L^\infty(\Rdual))p_i = \int_\Rdual^\oplus p_i(s)P_i p_i(s) ds.
\end{align*}
Also choose measurable fields of partial isometries $s \mapsto v_{i,g}(s)$ in $P_i$, and $s \mapsto W_g(s)$ in $P_1$, such that $v_{i,g} = \int_\Rdual^\oplus v_{i,g}(s)ds$ and $W_g = \int_\Rdual^\oplus W_g(s)ds$.

Fix a countable set of measurable fields $\mathfrak{X} = \{s \mapsto x(s)\}$, corresponding to a countable dense subset of $M_0$, such that for almost every $s$, the set $\{x(s) | x \in \mathfrak{X}\}$ is dense in $p_0(s)P_0 p_0(s)$, and such that for all $x \in \mathfrak{X}$ and all $g \in \Lambda$, also the measurable field $s \mapsto v_{0,g}(s) \alpha_g^0(x(s))v_{0,g}(s)^\star$ belongs to $\mathfrak{X}$.
Since $\theta : M_0 \to M_1$ is an isomorphism and by uniqueness of disintegrations, we find for almost all $s \in \Rdual$ an isomorphism $\theta_s : (p_0(s)P_0p_0(s), \vphi_0^{p_0(s)}) \to (p_1(s)P_1p_1(s), \vphi_1^{p_1(s)})$, such that
\begin{align}\label{eq:disint-theta}
\theta\left( \int_\Rdual^\oplus x(s)ds \right) = \int_\Rdual^\oplus \theta_s(x(s))ds \qquad \text{for all measurable fields $x \in \mathfrak{X}$}.
\end{align}
Combining \eqref{eq:conj-theta} and \eqref{eq:disint-theta}, we get that for all measurable fields $s \mapsto x(s)$ in $\mathfrak{X}$,
\begin{align*}
 \int_\Rdual^\oplus \theta_s\big( v_{0,g}(s) \alpha_g^0(x(s)) v_{0,g}(s)^\star \big) ds
 = \int_\Rdual^\oplus W_g(s) v_{1,g}(s) \alpha_g^1(\theta_s(x(s)))  v_{1,g}(s)^\star W_g(s)^\star ds.
\end{align*}
We then can find a conull subset $S \subset L\Rdual$ such that for all $s \in S$, $\theta_s$ is defined; $\{x(s)\mid x \in \mathfrak{X} \}$ is dense in $p_0(s)P_0p_0(s)$; $v_{i,g}(s)$ is a partial isometry with left support $p_i(s)$ and right support $\alpha_g(p_i(s))$, fixed by all $\{\sigma_t^{\vphi_i} \mid t \in \mathbb{Q}\}$; $W_g(s)$ is a unitary in $p_1(s) P_1 p_1(s)$ fixed by all $\{\sigma_t^{\vphi_1}\mid t \in \mathbb{Q}\}$; and such that for all $x \in \mathfrak{X}$,
\begin{align*}
\theta_s\big( v_{0,g}(s) \alpha_g^0(x(s)) v_{0,g}(s)^\star \big)
 = W_g(s) v_{1,g}(s) \alpha_g^1(\theta_s(x(s)))  v_{1,g}(s)^\star W_g(s)^\star.
\end{align*}
In particular, it follows for all $s \in S$ and for all $g \in \Lambda$, that $v_{i,g}(s) \in (P_i)_{\vphi_i}$, $W_g(s) \in (P_1)_{\vphi_1}$, and that
\begin{align*}
\theta_s \circ \Ad v_{0,g}(s) \circ \alpha_g^0 = \Ad W_g(s) \circ \Ad v_{1,g}(s) \circ \alpha_g^1 \circ \theta_s.
\end{align*}
Choosing some $s \in S$ for which $p_0(s)$ is nonzero, we obtain that the actions $\alpha^i : \Lambda\curvearrowright P_i$ are, up to reductions, cocycle conjugate through a state preserving isomorphism.
\end{proof}

\section{Proofs of \texorpdfstring{\cref{thmstar.A,corstar.B,corstar.C,corstar.D}}{Theorems A-D}}

\Cref{thmstar.A,corstar.B} follow now easily from \cref{thm.non-isomorphism-all-generalized-bernoulli}. Nevertheless, we give a detailed proof for the convenience of the reader.

\begin{proof}[Proof of \cref{thmstar.A}]
Let $\Lambda_i$ be icc groups in the class $\cC$, and $(P_i,\phi_i)$ be nontrivial amenable factors with normal faithful states such that $(P_i)_{\phi_i,\ap}$ are factors. Since the class $\cC$ does not contain amenable groups, the action $\Lambda_i \curvearrowright \Lambda_i$ has no invariant mean. It is obvious that every nontrivial $g \in \Lambda_i$ moves infinitely many points of $\Lambda_i$. The result follows now directly from \cref{thm.non-isomorphism-all-generalized-bernoulli} with $I_i = \Lambda_i$.  
\end{proof}

\begin{proof}[Proof of \cref{corstar.B}]
For $i=0,1$, let $\Lambda_i \in \cC$ be an icc group, and let $(P_i,\phi_i)$ be a nontrivial amenable factor with a normal faithful weakly mixing state $\phi_i$. Obviously, if $\Lambda_0 \cong \Lambda_1$ and the actions $\Lambda_i \curvearrowright (P_i,\phi_i)^{\Lambda_i}$ are conjugate, then the von Neumann algebras $P_i^{\Lambda_i} \rtimes \Lambda_i$ are isomorphic. Assume conversely that $(P_0,\phi_0)^{\Lambda_0} \rtimes \Lambda_0 \cong (P_1,\phi_1)^{\Lambda_1} \rtimes \Lambda_1$ are isomorphic. 

Since $\phi_i$ is weakly mixing, $\Delta_{\phi_i}$ has no eigenvalues, and hence we have that $(P_i)_{\phi,\ap} = \C$ for $i=0,1$. 
By \cref{lem.technical-condition} and putting $\vphi_i = \phi_i^{\Lambda_i}$, we then also get that $(P_i^{\Lambda_i})_{\vphi_i} = \C$. Applying \cref{thmstar.A}, we get that the groups $\Lambda_i$ are isomorphic, and that the actions $\Lambda_i \curvearrowright P_i^{\Lambda_i}$ are conjugate modulo the isomorphism $\Lambda_0 \cong \Lambda_1$, through a state preserving isomorphism.
\end{proof}

For the proof of \cref{corstar.C,corstar.D}, we recall the notion of a \emph{generalized 1-cocycle}.
Let $\alpha: G \to \Aut(M,\varphi)$ be an action of a locally compact group on a von Neumann algebra $(M,\varphi)$ with an n.s.f.~weight $\varphi$.
A \emph{generalized 1-cocycle} for $\alpha$ with \emph{support projection} $p \in M_\varphi$ is a continuous map $w:G \to M_\varphi$ such that $w_g \in p M_\varphi \alpha_g(p)$ is a partial isometry with $p=w_g w_g^\star$ and $\alpha_g(p) = w_g^\star w_g$, and
\begin{align*}
 w_{gh} = \Omega(g,h) w_g \alpha_g(w_h) \quad\text{ for all }g,h \in G,
\end{align*}
where $\Omega(g,h)$ is a scalar 2-cocycle.

\begin{proof}[Proof of \cref{corstar.C}]
For $i=0,1$, let $\Lambda_i$ be a direct product of two icc groups in the class $\cC$, and let $(P_i,\phi_i)$ be a nontrivial amenable factor with a normal faithful state $\phi_i$ such that $(P_i)_{\phi_i,\ap}$ is a factor. 
Assume that $(P_0,\phi_0)^{\Lambda_0} \rtimes \Lambda_0 \cong (P_1,\phi_1)^{\Lambda_1} \rtimes \Lambda_1$ are isomorphic. 
Putting $\vphi_i = \phi_i^{\Lambda_i}$, we get by \cref{thmstar.A} that there exists projections $p_i \in (P_i^{\Lambda_i})_{\vphi_i}$, such that the reductions of $\Lambda_i \curvearrowright P_i^{\Lambda_i}$ by $p_i$ are cocycle conjugate modulo the isomorphism $\Lambda_0 \cong \Lambda_1$ in a state preserving way. 
By amplifying both actions, we may assume that either $p_0 = 1$ or $p_1 = 1$. By interchanging $P_0$ and $P_1$ is necessary, we now assume that $p_0 = 1$. Identifying $\Lambda = \Lambda_0 = \Lambda_1$, writing $N_i = P_i^{\Lambda_i}$ and denoting the action $\Lambda \curvearrowright N_i$ by $\alpha^i$, 
this now means that there exists a state preserving isomorphism $\psi : N_0 \to p_1N_1p_1$ and a generalized 1-cocycle $(v_g)_{g \in \Lambda} \in (N_1)_{\vphi_1}$ with support projection $p_1$ for $\alpha^1 : \Lambda\curvearrowright N_1$, such that $\psi \circ \alpha^0_g = \Ad v_g \circ \alpha^1_g \circ \psi$ for all $g \in \Lambda$.

Put $M = (P_1)_{\phi,\ap}^{\Lambda_1}$ and note that $(N_1)_{\vphi_1} = M_{\vphi_1}$ by \cref{lem.technical-condition}. 
In particular, $v_g$ is a generalized 1-cocycle for the action $\Lambda \curvearrowright M$. Since the action $\Lambda \curvearrowright P_0^{\Lambda_0}$ has no nontrivial globally invariant subspaces, also $\Ad v_g \circ \alpha^1_g$ on $M$ has no such invariant subspaces, and it follows from \cite[Corollary 7.3]{vaes-verraedt;classification-type-III-bernoulli-crossed-products} that $v_g = \chi(g) v^\star \alpha_g^1(v)$ for all $g \in \Lambda$, where $\chi : \Lambda \to \T$ is a character and $v \in M_{\vphi_1,\lambda}$ satisfies $v^\star v = p_1$ and $vv^\star = 1$. Then $\Ad v \circ \psi : N_0 \to N_1$ is a state preserving isomorphism implementing a conjugation between the actions $\Lambda \curvearrowright P_i^{\Lambda_i}$.
\end{proof}

\begin{proof}[Proof of \cref{corstar.D}]
Let $\Lambda_i$ be icc groups in the class $\cC$, and $(P_i,\phi_i)$ be nontrivial amenable factors with normal faithful states such that $(P_i)_{\phi_i,\ap}$ are factors. We only need to show that if $P_i^{\Lambda_i} \rtimes \Lambda_i \times \Lambda_i$ are isomorphic, then the groups $\Lambda_i$ and the pairs $(P_i,\phi_i)$ are isomorphic. 
Assume now that $(P_0,\phi_0)^{\Lambda_0} \rtimes \Lambda_0 \times \Lambda_0 \cong (P_1,\phi_1)^{\Lambda_1} \rtimes \Lambda_1 \times \Lambda_1$ are isomorphic. Applying \cref{thm.non-isomorphism-all-generalized-bernoulli}, and proceeding exactly as in the proof of \cref{corstar.C}, using in particular \cite[Corollary 7.3]{vaes-verraedt;classification-type-III-bernoulli-crossed-products} and the weak mixingness of $\Lambda_0 \times \Lambda_0 \curvearrowright P_0^{\Lambda_0}$, we obtain an isomorphism $\delta : \Lambda_0 \times \Lambda_0 \to \Lambda_1 \times \Lambda_1$ and a state preserving isomorphism $\psi : P_0^{\Lambda_0} \to P_1^{\Lambda_1}$ satisfying $\psi \circ \alpha^0_g = \alpha^1_{\delta(g)} \circ \psi$ for all $g \in \Lambda_0 \times \Lambda_0$. Here we denoted the Bernoulli action $\Lambda_i \times \Lambda_i \curvearrowright P_i^{\Lambda_i}$ by $\alpha^i$.
An argument from \cite[Proof of Theorem 5.4]{popa-vaes;strong-rigidity-generalized-bernoulli-actions} (see also the last two paragraphs of the proof of \cite[Theorem C]{vaes-verraedt;classification-type-III-bernoulli-crossed-products}) now shows the desired result.
\end{proof}

\bibliography{bernoulli-shifts-structural-results}{}
\bibliographystyle{hmalpha}

\end{document}